\documentclass[letterpaper, 11pt]{article}

\usepackage{amsthm}
\usepackage{amssymb}
\usepackage{amsmath}
\usepackage[margin=1in]{geometry}
\usepackage{verbatim}

\newtheorem{prop}{Proposition}
\newtheorem{lemma}{Lemma}
\newtheorem{thm}{Theorem}
\newtheorem{cor}{Corollary}

\theoremstyle{definition}
\newtheorem{defn}{Definition}

\DeclareMathOperator{\ord}{ord}
\DeclareMathOperator{\Res}{Res}
\DeclareMathOperator{\ordRes}{ordRes}
\DeclareMathOperator{\MinResLoc}{MinResLoc}
\DeclareMathOperator{\PGL}{PGL}

\DeclareMathOperator{\CPA}{CPA}

\DeclareMathOperator{\Bary}{Bary}
\DeclareMathOperator{\supp}{supp}
\DeclareMathOperator{\GL}{GL}
\DeclareMathOperator{\diam}{\textrm{diam}}
\DeclareMathOperator{\RP}{RP}
\DeclareMathOperator{\GPR}{GPR}

\DeclareMathOperator{\diamG}{diam_{\zetaG}}

\newcommand{\pberk}{\textbf{P}^1_{\textrm{K}}}
\newcommand{\hberk}{\textbf{H}^1_{\textrm{K}}}
\newcommand{\aberk}{\textbf{A}^1_{\textrm{K}}}
\newcommand{\zetaG}{\zeta_{\text{G}}}

\newcommand{\del}{\partial}

\newcommand{\vv}{\vec{v}}
\newcommand{\vw}{\vec{w}}

\newcommand{\B}{\textrm{B}}

\title{Quantitative Logarithmic Equidistribution of the Crucial Measures}
\date{\today}
\author{Kenneth Jacobs}

\begin{document}
\maketitle

\begin{abstract}
Let $K$ be a algebraically closed field of characteristic 0 that is complete with respect to a non-Archimedean absolute value. Let $\phi\in K(z)$ with $\deg(\phi)\geq 2$. In this paper we establish uniform logarithmic equidistribution of the crucial measures $\nu_{\phi^n}$ attached to the iterates of $\phi$. These measures were introduced by Rumely in his study of the Minimal Resultant Locus of $\phi$. Our equidistribution result comes from a bound on the diameter of points in $\supp(\nu_{\phi^n})$ that depends only on $n$ and $\phi$. We also show that the sets $\MinResLoc(\phi^n)$ are bounded independent of $n$, and we give an explicit bound for the radius of a ball about $\zetaG$ containing $\Bary(\mu_\phi)$. 
\end{abstract}

\section{Introduction}

Let $K$ be an algebraically closed field that is complete with respect to a non-Archimedean absolute value $\lvert \cdot \rvert$. Denote by $\mathcal{O}$ its ring of integers and $\mathfrak{m}$ its maximal ideal. Let $k$ denote its residue field. 

Given a rational map $\phi\in K(z)$ of degree $d\geq 2$, Rumely \cite{Ru1, Ru2} introduced a continuous piecewise affine function $\ordRes_{\phi}(x)$ defined on the Berkovich line $\pberk$ over $K$ that carries information about the resultant of $\PGL_2(K)$-conjugates of $\phi$. Restricting this function to a canonical tree $\Gamma_{\textrm{Fix, Repel}}\subseteq \hberk$ and taking the Laplacian gives rise to a measure $\nu_{\phi}$ called the \emph{crucial measure}. The crucial measure $\nu_\phi$ is a discrete probability measure which can be written $$\nu_\phi = \frac{1}{d-1} \sum_{P\in \pberk} w_\phi(P) \delta_P \ ,$$ where the weight functions $w_\phi: \pberk \to \mathbb{R}_{\geq 0}$ take integer values and are determined by the reduction of $\phi$ at $P$. The collection of points with $w_\phi(P)>0$ is called the crucial set of $\phi$.

The first main result of this paper gives a quantitative bound on the location of points with $w_{\phi^n}(P) >0$ for some $n$:

\begin{thm}\label{thm:bounds}
Assume that $K$ has characteristic 0 and residue characteristic $p\geq 0$. Let $\phi\in K(z)$ have degree $d\geq 2$, and let $\mathcal{L}_\phi$ denote the Lipschitz constant for the action of $\phi$ on $\mathbb{P}^1(K)$ with respect to the chordal metric, and let $\tilde{\mathcal{L}}_\phi>\mathcal{L}_\phi\geq 1$. There exists a constant $N_0 = N_0(\phi)$ depending only on $\phi$ so that if $n\geq N_0$ and $P\in \hberk$ is a point with $w_{\phi^n}(P)>0$, we have $$\rho(P, \zetaG) < 3n \log_v \tilde{\mathcal{L}}_\phi\ .$$
\end{thm}

The proof of Theorem~\ref{thm:bounds} relies on two technical tools, Lemmas~\ref{lem:gausspreimageest} and~\ref{lem:critptfatou} below. The first builds on a result of Rumely and Winburn \cite{RW} to gives a lower bound for the distance between a root and a pole of $\phi$. The second relies on a modified lemma of Przytycki (see \cite{Prz}, Lemma 1) and the non-Archimedean Rolle's theorem established by Faber (see \cite{XF2} Application 1). We remark that Rumely and Winburn \cite{RW} have also given an upper bound for the Lipschitz constant $\mathcal{L}_{\phi, \textbf{d}}$ of $\phi$ with respect to the small metric $\textbf{d}_{\pberk}$ on $\pberk$. They show that $$\mathcal{L}_{\phi, \textbf{d}} \leq \max\left(\frac{1}{|\Res(\Phi)|^d}, \frac{d}{|\Res(\Phi)|}\right)\ ,$$ where $\Phi$ is a normalized lift of $\phi$.

The second main result in this paper uses the bound in Theorem~\ref{thm:bounds} to establish that the measures $\nu_{\phi^n}$ satisfy a logarithmic equidistribution condition. Let $\mu_\phi$ denote the equilibrium measure of $\phi$ supported on the Berkovich Julia set of $\phi$ (see \cite{FRLErgodic}). In \cite{KJ}, the author showed that the integrals $\int fd\nu_{\phi^n}$ converge to $\int fd\mu_\phi$ for functions $f\in \mathcal{C}(\pberk)$ that are continuous on $\pberk$. Here, we strengthen this result by establishing the following:

\begin{thm}\label{thm:logeq}
Assume that $K$ has characteristic 0 and residue characteristic $p\geq 0$. Fix a base point $\zeta_0\in \hberk$. Then for each $\zeta\in \pberk$,\begin{equation*}\left|\int \log_v \delta(z, \zeta)_{\zeta_0} d\left(\nu_{\phi^n} - \mu_\phi\right)(z)\right| = O\left(\frac{n}{d^n}\right)\end{equation*} Here, the big-O constant depends on the fixed base point $\zeta_0$, but is independent of $\zeta$.
\end{thm}
A more quantitative version of this result is given in Theorem~\ref{thm:quantlogeq} below where the error constant is given in terms of a constant $C_\phi$ depending only on $\phi$, the Lipschitz constant $\mathcal{L}_\phi$, and the H\"older constant and exponent of the potential function $u_\phi(\cdot, \zetaG)$.

As a corollary to Theorem~\ref{thm:logeq}, we obtain several uniform convergence results for potential-theoretic functions. Here, $\phi^\#$ dentoes the derivative derivative of $\phi$ with respect to the spherial metric on $\mathbb{P}^1(K)$, which extends continuously to $\pberk$.
\begin{cor}\label{cor:potentialfunctionconvergences}
Let $K$ be a complete, algebraically closed non-Archimedean valued field of characteristic 0. The crucial measures satisfy the following convergence properties:
\begin{enumerate}
\item For each fixed $\zeta_0\in \hberk$, the potential functions $u_{\nu_{\phi^n}}(z, \zeta_0)$ converge uniformly to $u_\phi(z, \zeta_0)$. 
\item The Arakelov-Green's functions $g_{\nu_{\phi^n}}(x,y)$ converge uniformly to $g_\phi(x,y)$.
\item The Lyapunov exponent of $\phi$ can be approximated in terms of the measures $\nu_{\phi^n}$: \begin{equation}\label{eq:lyapconv}\int_{\pberk} \log_v [\phi^\#] d\nu_{\phi^n} \to L_v(\phi):= \int_{\pberk} \log_v [\phi^\#] d\mu_\phi\ .\end{equation}
\end{enumerate}
\end{cor}

The third main result in this paper concerns the barycenter of the measures $\nu_{\phi^n}$ and $\mu_\phi$. Rivera-Letelier was the first to consider barycenters in dynamics over non-Archimedean fields. For a finite positive Radon measure $\nu$ on $\pberk$, the barycenter is the collection of points $P\in \pberk$ such that no direction contains more than half of the total mass of $\nu$, i.e. $$\Bary(\nu)  = \{ P \in \pberk \ : \ \nu(B_{\vv} (P)^-) \leq \frac{1}{2} \nu(\pberk)\ , \forall \vv\in T_P \}\ .$$

In \cite{Ru2}, Rumely showed that the barycenter of the crucial measure $\nu_\phi$ is precisely the Minimal Resultant Locus of $\phi$. More generally, the barycenter of a finite positive Radon measure $\nu$ on $\pberk$ is the collection of points where $g_\nu(x,x)$ is minimized. \\

In \cite{Ru1} Theorem 0.1, Rumely showed that the set $\MinResLoc(\phi)$ was contained in the $\rho$-ball $B_\rho\left(\zetaG, \frac{2}{d-1}\ordRes(\phi)\right)$. Generalizing this for iterates and using the convergence of $\MinResLoc(\phi^n)$ established in \cite{KJ}, we show

\begin{thm}\label{thm:barybounds}
Let $\phi\in K(z)$ be a rational function of degree $d\geq 2$, and let $R=\frac{2}{d-1} \ordRes(\phi)$. Then for each $n$, $$\Bary(\nu_{\phi^n}) = \MinResLoc(\phi^n) \subseteq B_\rho(\zetaG, R)\ .$$ If $m_0 = \min_{x\in \pberk} g_\phi(x,x)$ and $\mu_\phi$ dentoes the equilibrium measure of $\phi$, then $$\Bary(\mu_\phi) \subseteq B_\rho(\zetaG, R+m_0 - g_\phi(\zetaG, \zetaG))\ .$$
\end{thm}

\subsection{Outline}

In Section~\ref{sect:background} we recall the necessary background from dynamics and potential theory on $\pberk$, as well as the construction of the crucial measures $\nu_{\phi^n}$. Following this, in Section~\ref{sect:prelim} we present several preliminary technical lemmas; these results build on unpublished work of Rumely-Winburn and on a lemma of Przytycki. In Section~\ref{sect:bounds}, these bounds are used to establish Theorem~\ref{thm:bounds} by considering the various types of points that can receive weight. In Section~\ref{sect:logeq}, we prove a quantitative version of Theorem~\ref{thm:logeq} using the bounds in Section~\ref{sect:bounds} and the equidistribution in \cite{KJ} Theorem 4. Finally, in Section~\ref{sect:barybounds} we give a lemma estimating some of the coeffients of an iterate $\phi^n$; this together with explicit expressions for $\ordRes_{\phi^n}(x)$ and the convergence of these functions in \cite{KJ} Theorem 1 gives Theorem~\ref{thm:barybounds}.

\subsection{Acknowledgements}
The author would like to thank Robert Rumely for helpful conversations in the preparation of this article. The author was partially supported by a Research Training Grant DMS-1344994 of the RTG in Algebra, Algebraic Geometry, and Number Theory, at the University of Georgia.

\section{Background}\label{sect:background}
The background is divided into two sections. The first is a brief overview of the Berkovich projective line $\pberk$ and potential theory on $\pberk$, while the second recalls the construction of the crucial measures $\nu_{\phi}$. The book \cite{Ber} by Vladimir Berkovich gives a thorough development of analytic geometry over non-Archimedean fields. A rigorous development of dynamics on the Berkovich projective line can be found in the book by Baker and Rumely (\cite{BR}) and in the thesis of Juan Rivear-Letelier (see \cite{RL}). 
\subsection{The Berkovich Line}
Let $K$ be an algebraically closed field that is complete with respect to a non-Archimedean absolute value $|\cdot |_v$. We let $\mathcal{O} = \{z\in K\ : \ |z|\leq 1\}$ denote the ring of integers and $\mathfrak{m} =\{z\in K\ : \ |z|<1\}$. Let $k=\mathcal{O}/\mathfrak{m}$ denote the residue field of $K$, with $\textrm{char}(k) \geq 0$. 

Let $$D(a,r) := \{z\in K\ : \ |z-a|\leq r\}$$ denote the closed disc of radius $r$ about $a$ in the usual metric, and let $D(a,r)^-$ denote the corresponding open disc. Let $$B(a,r):= \{z\in \mathbb{P}^1(K)\ : \ ||z,a|| \leq r\}$$ denote the closed disc of radius $r$ about $a$ in the chordal metric, and let $B(a,r)^-$ denote the corresponding open disc. Note that $B(a,r) = \mathbb{P}^1(K)$ whenever $r\geq 1$, and $B(a,r) = D(a,r)$ whenever $|a|\leq 1$ and $r<1$.\\

The Berkovich affine line over $K$, denoted $\aberk$, is defined to be the collection of all multiplicative seminorms $[\cdot]$ on $K[X]$ that extend the absolute value on $K$. Berkovich \cite{Ber} has given a more intuitive description of $\aberk$ in terms of discs: most points $\zeta\in \aberk$ correspond to sup seminorms on discs $$[f]_{a,r} = \sup_{z\in D(a,r)} |f(z)|\ ,$$ where $a\in K$ and $r\in \mathbb{R}_{\geq 0}$. We write $\zeta=\zeta_{a,r}$ for the point corresponding to the disc $D(a,r)$. The unit disc is often referred to as the Gauss point, and we write $\zetaG = \zeta_{0,1}$. In the case $r=0$, these are simply the norms $[f]_{a,0} = |f(a)|$, which gives a natural embedding of $K$ into $\aberk$. Depending on whether or not $K$ is maximally complete, there may be additional points which correspond to limits of points $\zeta_{a_i, r_i}\in\aberk$ with $r_i \searrow 0$ but for which the intersection of the corresponding open discs $\bigcap D(a_i, r_i)^-$ is empty. These points will not play an important role in this paper.

We often categorize the points of $\aberk$ into four types:

\begin{itemize}
\item Type I points are of the form $\zeta_{a,0}$ for some $a\in K$.
\item Type II points are of the form $\zeta_{a,r}$ for some $r\in |K^\times|\setminus\{0\}$.
\item Type III points are of the form $\zeta_{a,r}$ for some $r\not\in |K^\times|$.
\item Type IV points are limits of points $\zeta_{a_i, r_i}\in \aberk$ for which $\bigcap D(a_i, r_i)^- = \emptyset$.
\end{itemize}

One can formally obtain the Berkovich projective line $\pberk$ by considering multiplicative seminorms on $K[X,Y]$ which extend $|\cdot|$. More intuitively, $\pberk$ is obtained from $\aberk$ by adjoining a type I point $\infty$, and so points in $\pberk\setminus \{\infty\}$ can still be considered as corresponding to discs in $K$. The collection of all type II, III and IV points in $\pberk $ is called Berkovich hyperbolic space and is denoted $\hberk$.

\subsubsection{Tree Structure}
There is a natural correspondence between intervals $[r,s]\in \mathbb{R}_{\geq 0}$ and segments $[\zeta_{a,r}, \zeta_{a,s}]$ in $\pberk$. Owing to the fact that the intersection of non-Archimedean discs is either empty or again a disc, this gives $\pberk$ a natural tree structure. In particular, it is uniquely path connected. We note that the type I and type IV points are the endpoints of the tree.

We can also discuss the tangent space of a point $\zeta\in \pberk$. If $\zeta$ is of type I or IV, then it is an endpoint and there is a unique direction pointing into $\pberk$. To understand the tangent space at a type II point $\zeta$, we first consider the special case of $\zeta=\zetaG$. Here, the underlying disc is $D(0,1)$, which can be written as a disjoint union $$D(0,1) = \bigsqcup_{a\in k} \alpha_a+\mathfrak{m}\ ,$$ where the $\alpha_a\in \mathcal{O}$ are coset representatives of $a\in k$. Each direction $\vv\in T_{\zetaG}$ corresponds either to one of the discs $a+\mathfrak{m}$ or to the direction pointing towards $\infty$, giving a natural correspondence between $T_{\zetaG}$ and $\mathbb{P}^1(k)$. By changing co\"ordinates, we have that $T_\zeta\cong \mathbb{P}^1(k)$ for any type II point $\zeta$. Type III points have two directions in their tangent space.

For each point $\zeta\in \pberk$, the connected components of $\pberk \setminus \{\zeta\}$ correspond to the various tangent directions away from $\zeta$. We label the connected components of $\pberk \setminus \{\zeta\}$ by $B_{\vv}(\zeta)^-$ (this is slightly different than \cite{BR}, where the authors used $B_\zeta(\vv)^-$ to denote the connected components of $\pberk \setminus \{\zeta\}$). These sets form a subbasis for the weak topology on $\pberk$. In this topology, $\pberk$ is compact but in general is not metrizable. Additionally, the points of type I and the points of type II each form a dense subset of $\pberk$. If $\pberk$ has points of type III, then they are dense in $\pberk$, and likewise if $\pberk$ has points of type IV they also form a dense subset of $\pberk$ in the weak topology.

Frequently, we will also consider finite subgraphs $\Gamma$ of $\hberk$; these are the union of finitely many segments $[P_i, Q_j]$ where each $P_i, Q_j$ is a point of type II or III. A function $f: \Gamma \to \mathbb{R}$ is said to be continuous and piecewise affine if there exists a finite set of points $\{s_1, ..., s_\ell\}\subseteq \Gamma$ such that each segment of $\Gamma\setminus\{s_1, ..., s_\ell\}$ is isometric to an open interval in $\mathbb{R}$, and $f$ is continuous on $\Gamma$ and affine on the segments in $\Gamma\setminus \{s_1, ..., s_\ell\}$. The collection of all continuous, piecewise affine functions on $\Gamma$ is denoted $\CPA(\Gamma)$. Given a point $\zeta\in \Gamma$, we let $T_\zeta(\Gamma) \subseteq T_\zeta$ denote the collection of directions that point into $\Gamma$. The valence of $\zeta$ in $\Gamma$ is the cardinality of $T_\zeta(\Gamma)$, and we write $v_\Gamma(\zeta):= \#T_\zeta(\Gamma)$.

\subsubsection{The Action of a Rational Map on $\pberk$}
The action of a rational map $\phi\in K(z)$ extends naturally to $\pberk$. This can be made formal using the seminorm construction on $K[X,Y]$ by choosing a lift $\Phi$ of $\phi$ and setting $[f]_{\phi(\zeta)} := [f\circ \Phi]_{\zeta}$ for all $f\in K[X,Y]$. 

As before, a more intuitive way to understand this action is by looking at discs: in non-Archimedean analysis, a holomorphic function will map a disc $D(a,r)$ to another disc $\phi(D(a,r)) = D(b,s)$ (\cite{ADS} Proposition 5.16). Taking more care, one can give an analogous statement for rational maps (which will map punctured discs map to punctured discs, where we puncture the domain at the poles and roots of $\phi$; see \cite{BR} Propositions 2.18 and 2.19). Informally, this gives $\phi(\zeta_{a,r}) = \zeta_{\phi(D(a,r))}$.

An important fact is that a rational map preserves the type of a point, e.g. if $\zeta$ is of type II, so too is $\phi(\zeta)$ (see \cite{BR} Proposition 2.15). 

We will often make use of the fact that the automorphism group of $\pberk$ is $\PGL_2(K)$ (see \cite{BR} Corollary 2.13); more precisely, given any triple $(a,\zeta, b)$, where $a,b\in \mathbb{P}^1(K)$ and a type II point $\zeta\in [a,b]$, then there exists a $\gamma\in \PGL_2(K)$ sending the triple $(a,\zeta,b)$ to $(0, \zetaG, \infty)$. Such a $\gamma$ need not be unique; however it is unique up to post-composition by an element $\tau\in \GL_2(\mathcal{O})$, which is the stabilizer of $\zetaG$ (\cite{Ru1}, Proposition 1.1).

\subsubsection{Reduction Types}

The rational map $\phi$ also induces a tangent map $\phi_*$ at each point of $\pberk$. At type I and type IV points, the tangent space is trivial and hence so is the action of $\phi_*$. If $\zeta$ is a type II point, we may choose elements $\tau, \gamma\in \PGL_2(K)$ so that $\tau\circ\phi\circ\gamma(\zetaG) = \zetaG$. Without loss of generality we may assume that $\zeta=\zetaG$ and that $\phi$ fixes $\zetaG$.

Write $$\phi(z) =\frac{f(z)}{g(z)} = \frac{a_d z^d + ... +a_0}{b_d z^d+...+b_0}\ ,$$ where $f$ and $g$ have no common factors. We say that $f$, $g$ are \emph{normalized representatives} of $\phi$ if $|a_i|, |b_i| \leq 1$ for each $i$ and at least one coefficient of $f$ or of $g$ is a unit. We can apply the reduction map $\tilde{\cdot} :\mathcal{O} \to k$ to each of the coefficients of a normalized representative and obtain a map $$\tilde{\phi} = \frac{\tilde{f}(z)}{\tilde{g}(z)} = \frac{\tilde{a_d}z^d+...+\tilde{a_0}}{\tilde{b_d}z^d+...+\tilde{b_0}}\in k(z)\ .$$ Note that $\tilde{f}, \tilde{g}$ may have a common factor, and so the degree of $\tilde{\phi}$ may be less than $d$; but because we are assuming $\phi(\zetaG) = \zetaG$, the map $\tilde{\phi}$ is non-constant (\cite{BR} Lemma 2.17). Indexing directions $\vv_a\in T_{\zetaG}$ by the corresponding element of $\mathbb{P}^1(k)$, we define the action $\phi_* \vv_a = \vv_{\tilde{\phi}(a)}$.\\

Fix now an arbitrary point $\zeta\in \hberk$. Choose $\gamma\in \PGL_2(K)$ with $\gamma(\zetaG) = \zeta$. By choosing a normalized representative of $\phi^\gamma$, we define the reduction of $\phi$ at $\zeta$ to be the map $\widetilde{\phi^\gamma}$. The degree of $\phi$ at $\zeta$ is defined to be $\deg_\phi(\zeta) = \deg(\widetilde{\phi^\gamma})$. While the reduction of $\phi$ at $\zeta$ depends on the choice of $\gamma$, the notion of degree is well-defined.

The reduction of a map $\phi$ at a point $\zeta$ determines much of its local behaviour. Perhaps most importantly, a point $\zeta$ is fixed by $\phi$ if and only if the reduction at $\zeta$ is non-constant (see \cite{BR} Lemma 2.17). Here, we recall a classification of type II fixed points based on the reduction of $\phi$ at $\zeta$:

\begin{itemize}
\item A point $\zeta$ is said to be a repelling fixed point if $\deg_\phi(\zeta) \geq 2$. Here, the map $\widetilde{\phi^\gamma}$ is conjugate over $\mathbb{P}^1(k)$ to a rational map of degree at least 2. A repelling periodic point is called a focused repelling point if there exists a unique direction $\vv\in T_\zeta$ containing all of the fixed points of $\phi$.
\item A point $\zeta$ is said to be a multiplicatively indifferent fixed point if $\deg_\phi(\zeta) = 1$ and the reduction $\widetilde{\phi^\gamma}$ is conjugate over $\mathbb{P}^1(k)$ to a map of the form $z\mapsto az$ for some $a\in k\setminus\{0, 1\}$. 
\item A point $\zeta$ is said to be an additively indifferent fixed point if $\deg_\phi(\zeta) = 1$ and the reduction $\widetilde{\phi^\gamma}$ is conjugate over $\mathbb{P}^1(k)$ to a map of the form $z\mapsto z+c$ for some $c\in k^\times$.
\item A point $\zeta$ is said to be an id-indifferent fixed point if $\deg_\phi(\zeta)=1$ and the reduction $\widetilde{\phi^\gamma}$ is the identity on $\mathbb{P}^1(k)$. 
\end{itemize}

These reduction types will play an important role in describing the measures $\nu_\phi$; see Section~\ref{sect:ordresphi} below.

\subsubsection{Metrics on $\pberk$}
In this section, we introduce two metrics $d$ and $\rho$ on $\pberk$. Both metrics generate the \emph{strong topology} on $\pberk$, which is finer than the weak topology introduced above. In this topology, $\pberk$ is no longer compact and in fact is not even locally compact! 

In order to define these two metrics, we need to introduce the notion of diameter relative to $\zetaG$ (one can define a diameter with respect to any $\zeta\in \pberk$; see \cite{BR} Chapter 4). If $x\in \pberk \setminus \{B_{\vv_\infty}(\zetaG)^-\}$ is a point of type I, II or III, then it corresponds to a (possibly degenerate) subdisc $D(a,r)$ in the closed unit disc $D(0,1)$. In this case we define $$\diam_{\zetaG}(x) := r\ .$$ If $x\in B_{\vv_\infty}(\zetaG)^-$ and $\psi(z) = \frac{1}{z}$, then $\psi(x)\in \pberk \setminus \{B_{\vv_\infty}(\zetaG)^-\}$, and we set $$\diam_{\zetaG}(x) := \diam_{\zetaG}(\psi(x))\ .$$ 

For a fixed base point $\zeta\in \pberk$, and any two points $z, w\in \pberk$, we will let $z\wedge_\zeta w$ denote the unique point in the intersection of the paths $[z,w]$, $[z, \zeta]$, and $[w, \zeta]$. We define the small metric on $\pberk$ by $$\textbf{d}_{\pberk}(x,y) := 2\diam_{\zetaG}(x\wedge_{\zetaG} y) - \diam_{\zetaG}(x) - \diam_{\zetaG}(y)\ .$$ The small model metric is an extension of twice the chordal metric on $\mathbb{P}^1(K)$, and is invariant under the action of $\GL_2(\mathcal{O})$. The action of a rational map $\phi$ on $\pberk$ is Lipschitz continuous with respect to this metric (see \cite{BR} Proposition 9.37). 

Similarly, we define the big metric on $\hberk$ by $$\rho(x,y) := 2\log_v \diam_{\zetaG}(x\wedge_{\zetaG} y) - \log_v \diam_{\zetaG}(x) - \log_v \diam_{\zetaG}(y)\ .$$ The big metric is $\PGL_2(K)$-invariant and will play an important role in the study of continuous piecewise affine functions on subgraphs $\Gamma\subseteq \hberk$. It is important to note that both metrics generate the strong topology and on $\hberk$ they are locally bounded in terms of one another.

\subsubsection{Potential Theory on $\pberk$}
We close this subsection with a discussion of potential theory on $\pberk$. Several different, but compatible, approaches to potential theory on $\pberk$ have been given (see \cite{BR}, \cite{FJ}, \cite{Th}). We will follow the approach of \cite{BR}.

Given a point $\zeta\in \pberk$ and a direction $\vv\in T_\zeta$, one can define the directional derivative of a function $f:\pberk \to \mathbb{R}$ as $$\del_{\vv} (f)(\zeta) = \lim_{t\to 0} \frac{f(\zeta+t\vv) - f(\zeta)}{t}\ , $$ provided the limit exists. For a fixed, finite graph $\Gamma\subseteq \hberk$ and a function $f\in \CPA(\Gamma)$, the directional derivatives exist for all $\zeta\in \Gamma$ and all $\vv\in T_\zeta(\Gamma)$. We define the Laplacian of $f$ on $\Gamma$ is the measure $$\Delta_\Gamma (f) := \sum_{\zeta\in \Gamma} -\left(\sum_{\vv\in T_\zeta(\Gamma)} \del_{\vv}(f)(\zeta)\right) \delta_\zeta\ .$$ This can be extended to a more general class of functions (those of bounded differential variation) defined on more general Borel sets in $\pberk$. \\

The fundamental kernel for potential theory on $\pberk$ is $$-\log_v \delta(x,y)_{\zeta}\ ,$$ where $\delta(x,y)_{\zeta}$ denotes the Hsia kernel relative to some fixed point $\zeta\in \pberk$. If $\zeta = \zetaG$, then this kernel is defined to be $$-\log_v \delta(x,y)_{\zetaG} := \rho(x\wedge_{\zetaG} y , \zetaG)\ .$$ For more general $\zeta$, we can define a generalized Hsia kernel as \begin{equation} \label{eq:Hsia}\delta(x,y)_{\zeta} = C_{\zeta}\frac{||x,y||}{||x, \zeta|| \cdot ||y, \zeta||} \end{equation} for an appropriately chosen $C_\zeta$ depending on $\zeta$. We remark that, for fixed $y, \zeta\in \pberk$, the function $h(x):= -\log_v \delta(x,y)_{\zeta}$ is linear along $[y, \zeta]$, and is constant on segments off of the path $[y, \zeta]$.\\

Let $\nu$ be a finite signed Radon measure on $\pberk$, and fix a point $\zeta\in \pberk$. The potential function associated to $\nu$ is defined to be $$u_\nu(z, \zeta):= -\int \log_v \delta(w, z)_{\zeta} d\nu(w)\ .$$ It satisfies the property that $\Delta u_\nu(\cdot, \zeta) =\nu- \nu(\pberk)\delta_\zeta$. We say that the measure $\nu$ has bounded potentials if, for some fixed $\zeta\in \hberk$, the function $u_\nu(z, \zeta)$ is bounded, and likewise that $\nu$ has continuous potentials if $u_\nu(z, \zeta) $ is continuous for some choice of $\zeta\in \hberk$. Using the transformation of the Hsia kernel given in (\ref{eq:Hsia}), one can show that these properties are independent of the choice of $\zeta$.

Finally, the Arakelov-Green's function attachted to $\nu$ is the two-variable function $$g_\nu(x,y):= -\int \log_v \delta(x,y)_{\zeta} d\nu(\zeta) + C\ ,$$ where the normalization constant is chosen so that $\iint g_\nu(x,y) d\nu(x) d\nu(y) = 0$. This function is symmetric, and for fixed $y\in \pberk$ we have $$\Delta g_\nu(\cdot, y) = \delta_y - \nu\ .$$

\subsection{The Function $\ordRes_{\phi}(x)$ and the Crucial Measures}\label{sect:ordresphi}
In this section we recall the necessary background that pertains to the measures $\nu_{\phi^n}$ and the crucial set. We refer the reader to the original papers (\cite{Ru1}, \cite{Ru2}) for a more rigorous development.

In \cite{Ru1}, Rumely studied a function $\ordRes_\phi:\pberk \to \mathbb{R}$ that measured the resultant of various $\PGL_2(K)$-conjugates of $\phi$. Given a map $\phi\in K(z)$, we say that a polynomial map $\Phi:K^2 \to K^2$ is a normalized representative of $\phi$ if $\Phi=[F,G]$ for homogeneous polynomials $F,G\in \mathcal{O}[X,Y]$ with $\phi(z) = \frac{F(z,1)}{G(z,1)}$, and at least one coefficient of $F$ or $G$ is a unit. Such a representative of $\phi$ is unique up to scaling by a unit $c\in k^\times$. 

The resultant $\Res(F,G)$ of a pair of homogeneous polynomials is the determinant of the Sylvester matrix defined by $F$ and $G$. It has the property that $\Res(F,G) = 0$ if and only if $F$ and $G$ have a common root over the algebraic closure. When studying the reduction $\widetilde{\phi}$ of $\phi$, it often happens that $\tilde{F}, \tilde{G}$ have a common factor over the residue field $k$. This is measured by the vanishing of $\Res(\tilde{F}, \tilde{G}) = \widetilde{\Res(F, G)}\ .$ On the other hand, $\tilde{F}, \tilde{G}$ have no common factors if and only if $$\ordRes(\phi) := -\log_v |\Res(F,G)| = 0$$ for some normalized lift $[F,G]$ of $\phi$. In this case, we say that the map $\phi$ has good reduction. Many maps, however, do not have good reduction. In this case, one can ask whether some $\PGL_2(K)$-conjugation $\phi^\gamma$ of $\phi$ has good reduction. If so, we say that $\phi$ has potential good reduction. \\

The function $\ordRes_\phi(x)$ was originally introduced to determine algorithmically whether a map $\phi$ had potential good reduction. The idea is to translate the problem from $\PGL_2(K)$ to $\pberk$. Recall that the points of $\PGL_2(K)$ act transitively on type II points; more precisely, given any type II point $\zeta\in \hberk$, there exists $\gamma\in \PGL_2(K)$ with $\gamma(\zetaG) = \zeta$. One is led to consider a function $$\ordRes_{\phi}(\zeta):= \ordRes(\phi^\gamma)\ .$$ Rumely shows (\cite{Ru1}, Theorem 0.1) that this function can be extended continuously to all of $\pberk$. The resulting function is continuous and piecewise affine along segments of $\pberk$. It is convex up and attains a minimum. The set on which $\phi$ attains its minimum is called the Minimal Resultant Locus, which we denote $\MinResLoc(\phi)$. $\MinResLoc(\phi)$ is either a point or a segment in $\hberk$ (\cite{Ru1}, Theorem 0.1), and Rumely showed further that it is contained in the tree $\Gamma_{\textrm{FR}}$ spanned by the type I fixed points of $\phi$ and the type II repelling fixed points in $\hberk$ (\cite{Ru2} Proposition 4.4). The map $\phi$ has potential good reduction in some co\"ordinate system if and only if $\ordRes_{\phi}(x)$ vanishes on $\MinResLoc(\phi)$. 

\subsubsection{ The Crucial Measures}
Given that $\ordRes_\phi(x)$ is continuous and piecewise affine, one can compute its Laplacian on finite subgraphs of $\pberk$. In particular, looking at an appropriately truncated version $\widehat{\Gamma_{\textrm{FR}}}$ of $\Gamma_{\textrm{FR}}$, one finds $$\Delta_{\widehat{\Gamma_{\textrm{FR}}}} \ordRes_\phi(\cdot) = 2(d-1) ( \mu_{\textrm{Br}} - \nu_{\phi})\ .$$ Here, the measure $\mu_{\textrm{Br}}$ is the canonical branching measure on $\widehat{\Gamma_{\textrm{FR}}}$ (see \cite{CR}) and $\nu_\phi$ is the crucial measure.\\

The crucial measures admit an alternate definition as a weighted sum of point masses. More precisely, we can write $$\nu_{\phi} = \frac{1}{d-1} \sum_{P\in \pberk} w_\phi(P) \delta_P\ .$$ For points $P\in \hberk$ fixed by $\phi$, the weight functions $w_\phi:\pberk \to \mathbb{Z}_{\geq 0}$ can be given in terms of the reduction of $\phi$ at $P$ and the number of shearing directions of $\phi$ at $P$: a direction $\vv\in T_P$ is called shearing if $B_{\vv}(P)^-$ contains a type I fixed point but $\phi_* \vv \neq \vv$. The number of shearing directions at $P$ is denoted $N_{\textrm{Shearing}}(P)$.

For points $Q$ which are not fixed by $\phi$, the weight functions $w_\phi$ depend on the valence of $Q$ in the tree $\Gamma_{\textrm{Fix}}$ spanned by the type I fixed points. Points of type I and of type IV are assigned weight 0.

\begin{defn}[\cite{Ru2}, Definition 8]\label{def:weights} For each $P\in \pberk$, the weight $w_\phi(P)$ is the following non-negative integer:
\begin{enumerate}
\item If $P\in \hberk$ is fixed by $\phi$, then $$w_\phi(P) = \deg_\phi(P) - 1 + N_\textrm{Shearing}\ .$$
\item If $P\in \hberk$ is a branch point of the tree $\Gamma_{\textrm{Fix}}$ spanned by the type I fixed points, then $$w_\phi(P) = \max(0, v_{\Gamma_{\textrm{Fix}}}(P) -2)\ .$$
\item If $P$ is a type I point, then $w_\phi(P) = 0$.
\end{enumerate}
\end{defn}

The tree in $\hberk$ spanned by the crucial set for the map $\phi$ is called the crucial tree, and is denoted $\Gamma_{\textrm{Cr}}$. The corresponding tree spanned by the crucial set for the map $\phi^n$ is $\Gamma_{\textrm{Cr}}^n$.

Rumely also gives the weight formula (\cite{Ru2} Theorem 6.2) $$\sum_{P\in \pberk} w_\phi(P) = d-1\ .$$ Hence $\nu_\phi$ is a probability measure supported at finitely many points in $\pberk$.

There is a relationship between the reduction types given above and the weights of fixed points: id-indifferent points receive no weight, as they have degree 1 and do not exhibit any shearing directions. Type II repelling periodic points always receive weight, since $\deg_\phi(P) \geq 2$. One can also show that additively indifferent points that lie in the tree $\Gamma_{\textrm{Fix}}$ must always receive weight. Multiplicatively indifferent points may or may not receive weight, depending on their valence in $\Gamma_{\textrm{Fix}}$.

\section{Preliminaries}\label{sect:prelim}
In this section, we establish some preliminary results that will be used in deriving the bounds on the crucial set. The first few results build on unpublished work of Rumely and Winburn and pertain to the distance between roots and poles of $\phi$. The latter lemmas make explicit a lemma of Przytycki and give a quantitative estimate of the proximity between critical points and periodic points.

\subsection{Bounds Concerning Roots and Poles}
The Lipschitz constant $\mathcal{L}_\phi$ for the action of $\phi$ on $\mathbb{P}^1(K)$ with respect to the chordal metric $||\cdot, \cdot||$ will play an important role in the results below. We recall that $$\mathcal{L}_\phi := \sup_{\substack{x,y\in \mathbb{P}^1(K)\\x\neq y}} \frac{||\phi(x), \phi(y)||}{||x,y||}\ .$$

\noindent\emph{Remark:} The Lipschitz constant is always at least 1: choose two points $x\neq y$ with $\phi(x) =0, \phi(y)=\infty$. Hence $||\phi(x), \phi(y)|| =2$, and $0<||x,y|| \leq 2$. In particular, $\mathcal{L}_\phi \geq \frac{||\phi(x), \phi(y)||}{||x,y||} = \frac{2}{||x,y||}\geq 1$. An upper bound for $\mathcal{L}_\phi$ is considered in work of Rumely and Winburn (see \cite{RW}), which will be discussed again below.

\noindent \emph{Remark:} The Lipschitz constant $\mathcal{L}_\phi$ is $\GL_2(\mathcal{O})$-invariant, in the sense that $\mathcal{L}_\phi = \mathcal{L}_{\phi^\tau}$ for any $\tau\in \GL_2(\mathcal{O})$. This follows from the fact that the chordal metric is $\GL_2(\mathcal{O})$-invariant. 

\begin{lemma}\label{lem:gaussnumberest}
Fix $\zeta_0\in \hberk$ of type II, and let $\zeta_1,... \zeta_k$ satisfy $\phi^n(\zeta_i) = \zeta_0$. Choose some $\gamma\in \PGL_2(K)$ with $\gamma(\zetaG) = \zeta_0$, and let $\mathcal{L}_{\phi^\gamma}$ be the Lipschitz constant for the action of $\phi$ on $\mathbb{P}^1(K)$ with respect to the chordal metric. Then for each $i = 1, ..., k$, we have $$\rho(\zeta_i, \zeta_0) \leq n\log_v\mathcal{L}_{\phi^\gamma}\ .$$
\end{lemma}

\begin{proof}
We may conjugate $\phi$ by an element $\gamma \in \PGL_2(K)$ that satisfies $\gamma(\zetaG) = \zeta_0$. While such a $\gamma$ is not unique, it is uniquely determined up to precomposition by an element $\PGL_2(\mathcal{O})$. Fix an index $i$, and let $r_i = \rho(\zeta_0, \zeta_i)$. We can find $\tau_i \in \PGL_2(\mathcal{O})$ so that $\gamma \circ \tau_i(\zeta_{0, r_i}) = \zeta_i$. Replacing $\phi$ by $\phi^{\gamma\circ \tau_i}$, we may assume $\zeta_0 = \zetaG$, and $\zeta_i = \zeta_{0,r}$ for some $r$. By the $\PGL_2(K)$-invariance of $\rho$, it is enough to estimate $\rho(\zeta_i, \zeta_0) = \rho(\zeta_{0,r}, \zetaG) = -\log_v r$. 

We use a description of the action of $\phi$ on $\pberk$ given in \cite{BR} Proposition 2.18. We can find $a_1, a_2, ..., a_k\in D(0, r)$ and $b_1, ..., b_s \in D(0,1)$ so that the image of the closed affinoid $D(0, r)\setminus \cup D(a_i, r)^-$ under $\phi^n$ is the closed affinoid $D(0,1)\setminus \cup D(b_k, 1)^-$. Choose two points $x,y\in D(0, 1)\setminus \cup D(b_k, 1)$ with $||x,y|| = 1$, and write $x=\phi^n(z), y=\phi^n(w)$ for $z,w\in D(0,r)\setminus \cup D(a_i, r)$. We find
\begin{align*}
1=|x-y| = ||x,y|| &\leq \mathcal{L}_\phi ||\phi^{n-1}(z), \phi^{n-1}(w)||\\
& \leq \dots\\
& \leq \mathcal{L}_\phi^n ||z,w||\\
& \leq \mathcal{L}_\phi^n r\ .
\end{align*} Thus we have the lower bound $$\mathcal{L}_\phi^{-n} \leq r\ .$$ In particular, $$\rho(\zetaG, \zeta_{0,r}) = \log_v\left(\frac{1}{r}\right) \leq \log_v \mathcal{L}_\phi^n = n\log_v \mathcal{L}_\phi\ .$$

Translating back to the original map $\phi$ and the original point $\zeta_0$ gives the assertion in the lemma.
\end{proof}

In their work on Lipschitz constants for $\phi$, Rumely and Winburn define the following constants:
\begin{itemize}
\item The root-pole number of $\phi$ is given $$\RP(\phi) = \min \{ ||\alpha, \beta|| \ : \ \alpha, \beta\in \mathbb{P}^1(K),\ \phi(\alpha) = 0,\ \phi(\beta) = \infty\}\ .$$
\item The Gauss preimage radius of $\phi$ is $$\GPR(\phi) = \min \{ \diam_{\zetaG} (x)\ : \ x\in \pberk, \phi(x) = \zetaG\}\ .$$
\item The Ball-mapping radius of $\phi$ is $$S_0(\phi) = \sup\{0<r\leq 1\ : \ \textrm{for all } a\in \mathbb{P}^1(K),\ \phi(B(a,r)^-) \neq \mathbb{P}^1(K)\}\ .$$
\end{itemize}

In their work, Rumely and Winburn show the following

\begin{prop}(Rumely-Winburn, \cite{RW}) Let $\phi\in K(z)$ have degree $d\geq 1$. Then $$0< \GPR(\phi) \leq \RP(\phi) \leq S_0(\phi)\leq 1\ .$$ \end{prop}

This, together with Lemma~\ref{lem:gaussnumberest} gives 
\begin{lemma}\label{lem:gausspreimageest}
Let $\phi\in K(z)$ have degree $d\geq 1$, and let $\mathcal{L}_\phi$ denote the Lipschitz constant for the action of $\phi$ on $\mathbb{P}^1(K)$ with respect to the chordal metric. Then $$\mathcal{L}_\phi^{-n}\leq \GPR(\phi^n) \leq \RP(\phi^n)\ .$$
\end{lemma}
\begin{proof}
The conclusion of Lemma~\ref{lem:gaussnumberest} gives that, for each $\zeta_i$ satisfying $\phi^n(\zeta_i) = \zetaG$, we have $$\rho(\zeta_i, \zetaG) \leq n\log_v \mathcal{L}_\phi\ .$$ Inserting this into the definition of $\diam_{\zetaG}(x)$ and taking the minimum now gives the result.
\end{proof}

A fact that will be used many times below is that if, for some $r<1$, $D(0,r)$ contains a root $\alpha$ and a pole $\beta$ for $\phi$, then we can bound $r$ below by $$\mathcal{L}_\phi^{-n} \leq ||\alpha, \beta|| = |\alpha-\beta| \leq r\ .$$

\subsection{Bounds Concerning Critical Points and Periodic Points}

We now prove two technical lemmas pertaining to the proximity of critical points and $n$-periodic points. The first is a modification of a lemma of Przytycki (see \cite{Prz} Lemma 1). Recall that we denote by $B(a,r)$ the closed disc of radius $r$ around $a$ with respect to the chordal metric. We will rely also on the fact that $B(a,r) = D(a,r)$ whenever $|a|\leq 1$ and $r<1$.

\begin{lemma}\label{lem:Przytycki}
Let $\mathcal{L}_\phi$ denote a Lipschitz constant for $\phi$ on $\mathbb{P}^1(K)$ with respect to the chordal metric. There exists a constant $0<A_\phi <1$ depending only on $\phi$ such that the following holds: for any critical point $c\in \mathbb{P}^1(K)$ of $\phi$ and any $n>0$, if $\epsilon < A_\phi \cdot \mathcal{L}_\phi^{-(n-1)}$ and $\phi^n(B(c,\epsilon)) \cap B(c, \epsilon) \neq \emptyset$ for some $n$, then $\phi^n(B(c, \epsilon)) \subseteq B(c, \epsilon)$.
\end{lemma}
\begin{proof}
We claim that we can conjugate $\phi$ by an element of $\PGL_2(\mathcal{O})$ so as to assume that $c=0$ and $|\phi(c)|\leq 1$. To see that this is possible, we consider two cases. If $c, \phi(c)$ lie in the same connected component $B_{\vv_a}(\zetaG)^-\subseteq\pberk \setminus \{\zetaG\}$, we can lift a M\"obius transformation $\widetilde{\gamma}\in \PGL_2(k)$ with $\gamma(a) = 0$ to a map $\gamma\in\PGL_2(\mathcal{O})$ which sends $B_{\vv_a}(\zetaG)^-$ to $B_{\vv_0}(\zetaG)^-$. Conjugating by $\gamma$ gives the desired configuration. If $c, \phi(c)$ lie in different connected components of $\pberk \setminus \{\zetaG\}$, then we can find an element of $\gamma\in \PGL_2(K)$ sending the triple $(c, \zetaG, \phi(c))$ to the triple $(0, \zetaG, 1)$ (see \cite{BR} Corollary 2.13); necessarily $\gamma$ fixes $\zetaG$, hence $\gamma\in \PGL_2(\mathcal{O})$, and conjugation by $\gamma$ achieves the desired configuration.\\

Writing $\phi$ as a Taylor series about $c=0$, we have $$\phi(z) = a_0 +a_2z^2 + ....\ .$$ For $k\geq 2$, let $a_k$ denote the first non-zero term in this expansion. We can find $\tilde{A}_\phi(c)<1$ depending only on $\phi,c$ so that, for $\epsilon < \tilde{A}_\phi(c)$ and $|z| < \epsilon <1$, we have 
\begin{equation}\label{eq:firstepsilonconstraint}
|\phi(z)-a_0| = |a_k|\cdot |z|^k < 1\ .
\end{equation}
In this case, since $|a_0|=|\phi(c)| \leq 1$, the inequality in (\ref{eq:firstepsilonconstraint}) implies that $|\phi(z)|\leq 1$ for all $z\in D(0, \epsilon)=B(0, \epsilon)$, and\begin{align*}
||\phi(z), a_0|| = |\phi(z) - a_0|&= |a_k|\cdot|z|^k  \\
&\leq |a_k| \cdot |z|^2\\
& < |a_k| \epsilon^2\ .
\end{align*} So, if $z,w\in B(c,\epsilon)=D(c,\epsilon)$ we find (using the ultrametric inequality at the last step)
\begin{align*}
||\Phi^n(z), \Phi^n(w)|| &\leq\mathcal{L}_\phi ||\Phi^{n-1}(z), \Phi^{n-1}(w)||\\
& \leq \dots \\
& \leq \mathcal{L}_\phi^{n-1} ||\Phi(z), \Phi(w)||\\
& \leq \mathcal{L}_\phi^{n-1}\cdot |a_k|\epsilon^2\ .
\end{align*}

Let $A_\phi(c):=\min(\tilde{A}_\phi(c), \frac{1}{|a_k|})<1$; then by the arguments above if $\epsilon < A_\phi(c)\cdot \mathcal{L}_\phi^{-(n-1)}$ we find $$||\Phi^n(z), \Phi^n(w)|| \leq \mathcal{L}_\phi^{n-1}\cdot |a_k|\epsilon^2< \epsilon\ .$$

In particular, the chordal diameter of $\phi^n(B(c,\epsilon))$ is bounded above by $\epsilon$, hence the condition $\phi^n(B(c, \epsilon)) \cap B(c, \epsilon) \neq \emptyset$ implies (by the ultrametric inequality) $\phi^n(B(c, \epsilon)) \subseteq B(c, \epsilon)$.  Letting $A_\phi := \min_{c \textrm{ is a critical point}} A_\phi(c)$ gives the required constant. 

\end{proof}

The preceeding lemma can be thought of as a quantitative expression of the fact that if a critical point $c$ of $\phi$ is very close to an $n$-periodic point, then both must lie in the Fatou set. The next lemma will give a similar relationship for a critical point \emph{of the n-th iterate} $\phi^n$ and an $n$-periodic point. The idea is that if $B(c,r)$ contains a critical point $c$ of $\phi^n$ and an $n$-periodic point $f$, then some image $\phi^j B(c,r)$ will contain the critical point $\phi^j(c)$ of $\phi$ along with the $n$-periodic point $\phi^j(f)$ of $\phi$. We then translate the quantitative results from the preceeding lemma to $\phi^j(B(c, r))$.

\begin{lemma}\label{lem:critptfatou}
Let $\mathcal{L}_\phi$ denote the Lipschitz constant for the action of $\phi$ on $\mathbb{P}^1(K)$ with respect to the chordal metric, and let $B_\phi:=\min(1, A_\phi)$ where $A_\phi$ is the constant in Lemma~\ref{lem:Przytycki}. Let $n\geq 1$. Suppose that  $D(0,r)$ contains an $n$-periodic point of $\phi$ and a critical point of $\phi^n$. If $r < B_\phi \cdot\mathcal{L}_\phi^{-2(n-1)}$, then $D(0,r) \subseteq \mathcal{F}(\phi)$.
\end{lemma}

\begin{proof}
It will be more convenient to work with balls in the chordal metric. Let $f\in D(0,r)$ denote an $n$-periodic point of $\phi$, and consider the sets $D(0, r)=B(0,r), \phi(B(0,r)), \phi^2(B(0,r)), ..., \phi^{n-1}(B(0,r))$. Since $\phi^n$ has a critical point in $D(0,r)$, the map $\phi$ must have a critical point in some $\phi^j(B(0,r))$, $j=0,1,...,n-1$. As in the proof of Lemma~\ref{lem:gaussnumberest}, we can estimate $$||\phi^j(z), \phi^j(w)|| \leq \mathcal{L}_\phi^j ||z,w|| \leq \mathcal{L}_\phi^j\cdot  r\ , \ \forall z,w\in D(0,r) = B(0,r)\ .$$

Therefore $\phi^j(B(0,r))\subseteq B(\phi^j(0), \epsilon)$ where $\epsilon= \mathcal{L}_\phi^j\cdot r$. We note that $\epsilon < 1$: the constant $r$ was chosen so that $r<B_\phi\cdot  \mathcal{L}_\phi^{-2(n-1)}$; since $\mathcal{L}_\phi \geq 1$, this implies $r<\mathcal{L}_\phi^{-2(n-1)}<\mathcal{L}_\phi^{-j}$ for each $j=0, 1, ..., n-1$. Hence $\epsilon=\mathcal{L}_\phi^j r<1$.

In particular, $B(\phi^j(0), \epsilon)$ is a chordal disc containing a critical point of $\phi$ and, since it contains the $n$-periodic point $\phi^j(f)$, it satisfies $\phi^n(B(\phi^j(0),\epsilon)) \cap B(\phi^j(0), \epsilon) \neq \emptyset$. Thus if $r$ satisfies $$\mathcal{L}_\phi^j r< B_\phi \cdot \mathcal{L}_\phi^{-(n-1)}\ ,$$ then Lemma~\ref{lem:Przytycki} ensures that $\phi^n(B(\phi^j(0), \epsilon))\subseteq B(\phi^j(0), \epsilon) \subseteq \mathcal{F}(\phi)$. In particular, $B(\phi^j(0), \epsilon) \subseteq \mathcal{F}(\phi)$, and so $$D(0,r) = B(0,r) \subseteq \phi^{-j} (\phi^j(B(0,r))) \subseteq \mathcal{F}(\phi)\ .$$

\end{proof}

\section{Bounds for Weighted Points}\label{sect:bounds}
In this section, we establish the following theorem:

\begin{thm}\label{thm:boundoncrucialset}
Let $K$ be a complete, algebraically closed non-Archimedean valued field of characteristic 0. Let $\phi\in K(z)$ have degree $d\geq 2$.
\begin{itemize}
\item[(A)] Suppose $\phi$ has potential good reduction, and let $P\in \hberk$ be the point at which $\phi$ attains good reduction. Let $\Phi$ be a normalized lift of $\phi$ at $\zetaG$. Then $$\rho(P, \zetaG) \leq \frac{2}{d-1} \log_v |\Res(\Phi)|^{-1}\ .$$ 

\item[(B)] Suppose $\phi$ does not have potential good reduction, and let $\mathcal{L}_\phi$ denote the Lipschitz constant for the action of $\phi$ on $\mathbb{P}^1(K)$ with respect to the chordal metric. Then there exists a constant $B_\phi>0$ depending only on $\phi$ such that the following holds: Suppose that for some $n\geq 1$ and some $P\in \hberk$, we have $w_{\phi^n}(P)>0$. Then \begin{align}\label{eq:maxnotgoodreduction}
\rho(P, \zetaG) \leq \max&\left(n \log_v \mathcal{L}_\phi, 2(n-1)\log_v \mathcal{L}_\phi - \log_v B_\phi +\frac{1}{p-1} \right)\ .
\end{align}
\end{itemize}
\end{thm}

We will establish this theorem by considering separately the different types of weighted points. Our first step is to address those maps that have good reduction; this is essentially a restatement of \cite{Ru1} Theorem 0.1:

\begin{prop}\label{prop:goodreductionbound}
Let $\Phi$ be a normalized lift of $\phi$. If $\phi$ has potential good reduction, and $P=\zeta$ is the point at which $\phi$ attains good reduction, then $$\rho(P, \zetaG) \leq \frac{2}{d-1} \log_v |\Res(\Phi)|^{-1}\ .$$
\end{prop}
\begin{proof}
If $\phi$ has good reduction, so too does $\phi^n$ for all $n$ (see, e.g., \cite{BR} Theorem 10.17). In this case, the crucial set consists of a single point $\zeta$, which is also the Minimal Resultant Locus of $\phi$. In \cite{Ru1} Theorem 0.1, Rumely established that the Minimal Resultant Locus lies in the ball $B_\rho(\zetaG, \frac{2}{d-1} \ordRes(\phi)) = B_\rho(\zetaG, -\frac{2}{d-1}\log_v |\Res(\Phi)|)$, where $\Phi$ is a normalized lift of $\phi$. Hence the asserted bound holds.
\end{proof}

We are left to consider the case when $\phi$ does not have potential good reduction. Here, we proceed by obtaining bounds for the different types of points appearing in the crucial set.

\begin{prop}\label{prop:focusedrepellingbound}
Let $\Phi$ be a normalized lift of $\phi$, and let $P$ be a focused repelling fixed point for some iterate $\phi^n$. Then $$\rho(P, \zetaG) \leq n \log_v \mathcal{L}_\phi\ .$$
\end{prop}

\begin{proof}
If $P=\zetaG$, the assertion is clear, and so we assume $P\neq \zetaG$. Let $\vv_a\in T_{\zetaG}$ be the direction pointing towards $P$, and let $\vv_b\in T_P$ be a direction pointing away from $\zetaG$. Choose a type I point $S\in B_{\vv_b}(P)^-$. Choose $\gamma\in \PGL_2(\mathcal{O})$ so that $\gamma(S) = 0$; then $\gamma(B_{\vv_a}(\zetaG)^-) = B_{\vv_0}(\zetaG)^-$, and $P= \gamma^{-1}(\zeta_{0,r})$, where $\rho(\zetaG, P) = -\log_v r$. Replacing $\phi$ by $\phi^\gamma$, we may assume $P=\zeta_{0,r}$. It suffices to find an upper bound on $-\log_v r$. We consider two cases:
\begin{itemize}
\item[Case 1:] Suppose that the direction $\vv_\infty\in T_P$ is the direction pointing into $\Gamma_{\textrm{Fix, Repel}}$. By Rumely's Tree Intersection theorem (\cite{Ru2}, Theorem 4.2) we have that $$\Gamma_{\textrm{Fix, Repel}}^n = \bigcap_{b\in \mathbb{P}^1(K)} \Gamma_{\textrm{Fix}, (\phi^n)^{-1}(b)}\ .$$ Hence $P\in \Gamma^n_{\textrm{Fix}, (\phi^n)^{-1}(0)} \cap \Gamma^n_{\textrm{Fix}, (\phi^n)^{-1}(\infty)}$. But since $P$ is a focused repelling periodic point, it does not lie in $\Gamma^n_{\textrm{Fix}}$, and therefore there must be both a pole $\beta$ and a root $\alpha$ of $\phi^n$ in $\pberk \setminus B_{\vv_\infty}(P)^-$, hence in $D(0,r)$. In particular, $$||\alpha, \beta|| = |\alpha-\beta| \leq r\ .$$ Lemma~\ref{lem:gausspreimageest} now gives $$r \geq \RP(\phi) \geq \mathcal{L}_\phi^{-n}\ .$$ After taking logarithms, this is the asserted bound.

\item[Case 2:] Suppose that some finite direction $\vv_a\in T_P\setminus \{\vv_\infty\}$ is the direction pointing into $\Gamma_{\textrm{Fix, Repel}}$. By \cite{Ru2} Proposition 3.1(B), we know $s_{\phi^n}(P, \vv_a) >0$, and hence $\phi^n(B_{\vv_a}(P)^-) = \pberk$. In particular, $\phi^n$ has a root $\alpha$ and a pole $\beta$ in $B_{\vv_a}(P)$. In particular, $\alpha, \beta\in D(a, r)$, and so $$||\alpha, \beta|| = |\alpha-\beta| \leq r\ ,$$ and arguing as in the previous case gives $$r\geq \mathcal{L}_\phi^{-n}\ .$$ Taking logarithms, this is the asserted bound.
\end{itemize}

\end{proof}

\begin{prop}\label{prop:weightedremotebound}
Assume that $K$ is a complete, algebraically closed non-Archimedean valued field with characteristic 0. Let $\mathcal{L}_\phi$ denote the Lipschitz constant for the action of $\phi$ on $\mathbb{P}^1(K)$ with respect to the chordal metric. Let $P$ be a point with $w_{\phi^n}(P)>0$ that is fixed by $\phi^n$ and which is not a focused repelling periodic point. Let $B_\phi$ be the constant in Lemma~\ref{lem:critptfatou}. Then \begin{equation}\label{eq:maximum}\rho(P, \zetaG)\leq \max\left( n\log_v \mathcal{L}_\phi, 2(n-1) \log_v \mathcal{L}_\phi - \log_v B_\phi+\frac{1}{p-1}\right) \ .\end{equation} If $P$ has a shearing direction, then $\rho(P, \zetaG) \leq n \log_v \mathcal{L}_\phi$.
\end{prop}

\begin{proof}
If $P = \zetaG$ then the assertion is clear, so assume $P\neq \zetaG$. Since $P$ is not a focused repelling periodic point, we can find two distinct directions $\vv_a, \vv_b\in T_P(\Gamma^n_{\textrm{Fix}})$ containing type I $n$-periodic points $a, b$ (resp.). Without loss of generality we can assume $\vv_a\neq \vv_{\zetaG}$. 

Let $\vv_c\in T_{\zetaG}$ be chosen so that $P\in B_{\vv_c}(\zetaG)^-$, and let $\tilde{\gamma}\in \PGL_2(k)$ be a map such that $\tilde{\gamma}(\tilde{c}) = 0$. Let $\gamma\in \PGL_2(\mathcal{O})$ be a lift of $\tilde{\gamma}$ with $\gamma(a) = 0$. Then $P = \gamma^{-1}(\zeta_{0,r})$ where $-\log_v r = \rho(\zetaG, P)$, and $0$ is a fixed point for $\phi^\gamma$. Replacing $\phi$ by $\phi^\gamma$, the $\PGL_2(K)$-invariance of $\rho$ implies that it suffices to estimate $\rho(\zetaG, \zeta_{0,r}) = -\log_v r$. We will further assume that $r< \gamma_p^{-1}$, where $\gamma_p = |p|^{-1/(p-1)}$, for otherwise $r\geq \gamma_p^{-1}$ implies $\rho(\zetaG, \zeta_{0,r}) = -\log_v r \leq \frac{1}{p-1}$, which is no larger than the second term in the maximum appearing in (\ref{eq:maximum}). \\

Since $w_{\phi^n}(P) >0$, $P$ is not id-indifferent. Thus for every $\vv\in T_P(\Gamma^n_{\textrm{Fix}})$, we have (see \cite{Ru2} Lemma 2.1) $$\#F_{\phi^n}(P, \vv) = s_{\phi^n}(P, \vv) + \#\tilde{F}_{\phi^n}(P, \vv)\ .$$ We now consider two cases:

\begin{itemize}

\item[Case 1:] Suppose that $s_{\phi^n}(P, \vv) >0$ for some $\vv\in T_P(\Gamma^n_{\textrm{Fix}})\setminus \{\vv_\infty\}$. Then $\phi^n(B_{\vv}(P)^-) = \pberk$, and hence $B_{\vv}(P)^-$  contains both a pole $\beta$ and a root $\alpha$ of $\phi^n$. Arguing as in the proof of Proposition~\ref{prop:focusedrepellingbound}, we have $$\mathcal{L}_\phi^{-n} \leq r\ .$$ 

\item[Case 2:] Otherwise, $s_{\phi^n}(P, \vv_a) = 0$ for all $\vv_a \in T_P(\Gamma^n_{\textrm{Fix}})\setminus\{\vv_\infty\}$. We remark here that this implies $\phi^n_* \vv_a = \vv_a$ for all $\vv_a\in T_P(\Gamma^n_{\textrm{Fix}})\setminus\{\vv_\infty\}$ (see \cite{Ru2}, Lemma 2.1). Since $P$ is not a point of good reduction, Faber's theorem (\cite{XF} Lemma 3.17) implies that \emph{some} direction has $s_{\phi^n}(P, \vv) >0$, and we conclude $s_{\phi^n}(P, \vv_\infty)>0$. By \cite{Ru2} Lemma 2.1, this implies that $B_{\vv_\infty}(P)^-$ contains a type I fixed point of $\phi^n$. We now consider two subcases:\\

\begin{itemize}
\item[Case 2A:] We claim that, if $\phi^n_* \vv_\infty \neq \vv_\infty$, then there is a pole of $\phi^n$ in $D(0,r)$: the map $\phi^n_*: T_P \to T_P$ is surjective, and so if $\phi^n_* \vv_\infty \neq \vv_\infty$, then there is a finite direction $\vv_a$ with $\phi^n_* \vv_a = \vv_\infty$. We necessarily have that $\phi^n(B_{\vv_a}(P)) \supseteq B_{\vv_\infty}(P)$, whereby $B_{\vv_a}(P)$ contains a pole $\beta$ of $\phi^n$. In particular, $$||\beta, 0|| = |\beta - 0| \leq r\ .$$ Arguing as above, we have $$\mathcal{L}_\phi^{-n} \leq r\ .$$ Taken together, Cases 1 and 2A imply that if $P$ has a shearing direction, then $\mathcal{L}_\phi^{-n}\leq r$, which is the final assertion of the lemma.\\

\item[Case 2B:] We are thus left to the case that $\phi^n_* \vv = \vv$ for each direction $\vv\in T_P(\Gamma^n_{\textrm{Fix}})$, and $s_{\phi^n}(P, \vv_\infty) >0$ while $s_{\phi^n}(P, \vv_a) = 0$ for all $\vv_a\in T_P(\Gamma^n_{\textrm{Fix}})\setminus\{\vv_\infty\}$. Note that such a $P$ cannot be an additively indifferent or multiplicatively indifferent point: such points have degree 1, and so by the weight formulae (see Definition~\ref{def:weights}) they must have a shearing direction in order to receive weight. Thus, in this case:

\begin{center}
 $P$ is a \emph{repelling} $n$-periodic point. ($\ast$)\\
\end{center}

We claim that, after conjugating by an element $\gamma\in \PGL_2(\mathcal{O})$, we can assume that (1) $\phi^n(0) = 0$, (2) $(\phi^n)_*\vv_0 = \vv_0$, and (3) there exists some $\vv_a\in T_P\setminus\{\vv_\infty, \vv_0\}$ with $(\phi^n)_*\vv_a = \vv_0$. Note that condition (1) is satisfied by our initial conjugation, and (2) is therefore satisfied since we are assuming that there is no shearing. It remains to show that (3) can be obtained in a way that preserves (1) and (2).

Note that since $\phi^n$ has no poles in $D(0,r)$, we must have that $(\phi^n)_* \vw = \vv_\infty$ implies $\vw = \vv_\infty$. Hence the reduction $\widetilde{\phi^n}$ at $P$ is a polynomial map. If condition (3) fails, then the only preimage of $0$ under $\widetilde{\phi^n}$ is again zero, hence $\widetilde{\phi^n} = z^{\tilde{d}}$ where $\tilde{d} = \deg_{\phi^n}(P)$. This polynomial has non-trivial finite fixed points $\{a_1, a_2,...,a_\ell\}\in \mathbb{P}^1(k)$, which correspond to directions with $\widetilde{\#F}_{\phi^n}(P, \vv_{a_i}) > 0$; moreover, for each $a_i \neq 0 $ we can find at least one $b\in \mathbb{P}^1(k)\setminus \{a_1, ..., a_\ell\}$ with $\widetilde{\phi^n}(b) = a_i$. 

Since $\widetilde{\#F}_{\phi^n}(P, \vv_{a_i}) >0$, \cite{Ru2} Lemma 2.1 implies that $B_{\vv_{a_i}}(P)^-$ contains a type I $n$-periodic point $f_i$. Further, $\phi^n_* \vv_{b} = \vv_{a_i}$ for $b$ chosen as above. Conjugating $\phi^n$ by $\gamma(z) = z+f_i$ for any fixed $i\in \{1, 2,.., k\}$ will give a map satisfying (1), (2) and (3).\\

With this conjugation, we find that $\phi^n(0) = 0$ and that there is some non-zero direction $\vw = \gamma^{-1}_*(\vv_{a_i})$ with $\phi^n(B_{\vw}(P)^-) = B_{\vv_0}(P)^-$ (recall that $\vv_\infty$ is the only direction with surplus multiplicity). In particular, $B_{\vw}(P)^-$ contains a non-zero root of $\phi^n$. By the non-Archimedean Rolle's theorem (\cite{XF2} Application 1), there is a critical point of $\phi^n$ in $D(0, r\cdot \gamma_p)$, where $\gamma_p = |p|^{-1/(p-1)}>1$. 

We now apply Lemma~\ref{lem:critptfatou}. Let $P' = \zeta_{0, r\cdot \gamma_p}$, and $\vv_0'\in T_{P'}$ the direction towards $0$. The classical disc $D(0,r\cdot \gamma_p)$ contains an $n$-periodic point $P$ of $\phi$ and a critical point of $\phi^n$. If $r\cdot \gamma_p < B_\phi \mathcal{L}_\phi^{-2(n-1)}$, then $D(0, r\cdot \gamma_p) \subseteq \mathcal{F}(\phi)$. In particular, $B_{\vv_0'}(P')^- \subseteq \mathcal{F}(\phi)$; but by ($\ast$), $P\in \B_{\vv_0'}(P')^-$ is a repelling periodic point, and thus lies in the Julia set, which is a contradiction. So $r\cdot \gamma_p \geq B_\phi \mathcal{L}_\phi^{-2(n-1)}$. Moving the $\gamma_p$ to the other side of the inequality and taking logarithms, this gives the asserted bound.

\end{itemize}
\end{itemize}

\end{proof}

Lastly, we bound the distance from $\zetaG$ to weighted points that are branch points of $\Gamma_{\textrm{Fix}}^n$ which are moved by $\phi^n$:
\begin{prop}\label{prop:weightedbranchpt}
Assume that $K$ is a complete, algebraically closed non-Archimedean valued field with characteristic 0. Let $B_\phi$ be the constant from Lemma~\ref{lem:critptfatou}, and let $\gamma_p = |p|^{-1/(p-1)}$. Let $P$ be a point with $w_{\phi^n}(P)>0$ that is moved by $\phi^n$. Then $P$ is necessarily a branch point of $\Gamma^n_{\textrm{Fix}}$, and $$\rho(P, \zetaG) \leq \min\left(n\log_v \mathcal{L}_\phi, 2(n-1) \log_v \mathcal{L}_\phi - \log_v B_\phi+\frac{1}{p-1}\right) \ .$$ 
\end{prop}

\begin{proof}
 If $P=\zetaG$, the result is clear, and so we may assume $P\neq \zetaG$. By Rumely's classification of points with $w_{\phi^n}(P)>0$ (see \cite{Ru2} Proposition 6.1), $P$ must be a branch point of $\Gamma^n_{\textrm{Fix}}$ which is moved by $\phi^n$. 

We normalize $\phi$ as in the preceeding proposition, so that $0$ is an $n$-periodic point and $P=\zeta_{0,r}$ for some $r<1$: since $P$ is a branch point of $\Gamma^n_{\textrm{Fix}}$, we can find at least two directions $\vv_a, \vv_b\in T_P(\Gamma^n_{\textrm{Fix}})$ that contain type I $n$-periodic points $a,b$ (resp.). Without loss of generality, we may assume that $\vv_a\neq \vv_{\zetaG}$. Let $\vv_c\in T_{\zetaG}$ be chosen so that $P\in B_{\vv_c}(\zetaG)^-$, and let $\tilde{\gamma}\in \PGL_2(k)$ be a map such that $\tilde{\gamma}(c) = 0$. Let $\gamma\in \PGL_2(\mathcal{O})$ be a lift of $\tilde{\gamma}$ with $\gamma(a) = 0$. Then $P = \gamma^{-1}(\zeta_{0,r})$ where $-\log_v r = \rho(\zetaG, P)$, and $0$ is a fixed point for $(\phi^{(n)})^\gamma$. Replacing $\phi$ by $\phi^\gamma$, we can assume that $P=\zeta_{0,r}$. The $\PGL_2(K)$-invariance of $\rho$ implies that it suffices to estimate $\rho(\zetaG, \zeta_{0,r}) = -\log_v r$. As in the previous proposition, we can further assume that $r< \gamma_p^{-1}$, where $\gamma_p = |p|^{-1/(p-1)}$.\\

If $\phi^n$ has a pole $\beta$ in $D(0,r)$, then $|\beta| \leq r$; since $0$ is a root of $\phi^n$, we may argue using the root-pole number as in the previous proposition to find $$\mathcal{L}_\phi^{-n} \leq ||0, \beta||  = |\beta| \leq r\ .$$

Suppose instead that $\phi^n$ has no poles in $D(0,r)$. Then for each finite direction $\vv_a\in T_P\setminus\{\vv_\infty\}$, we have $\infty\not\in \phi^n(B_{\vv_a}(P)^-)$. In particular, $\phi^n(B_{\vv_a}(P)^-) \neq \pberk$, and so $\phi^n(B_{\vv_a}(P)^-)$ must be a generalized Berkovich disc $B_{\phi^n_* \vv_a}(\phi^n(P))^-$ (see \cite{RL} Lemma 2.1, or also \cite{BR} Proposition 9.41). Let $Q=\phi^n(P)$.

We first claim that $Q\in (P, \infty]$. If not, let $\vw_{QP}\in T_Q$ be the direction at $Q$ pointing towards $P$. Then $\infty\in B_{\vw_{QP}}(Q)^-$. For each finite direction $\vv_a\in T_P(\Gamma^{n}_{\textrm{Fix}})$ containing a fixed point, \cite{Ru2} Lemma 2.2 implies that either $Q\in B_{\vv_a}(P)^-$ or $P\in B_{(\phi^n)_*\vv_a}(Q)^-$. The first condition can hold for at most one $\vv_a\in T_P(\Gamma^{n}_{\textrm{Fix}})$, and since $P$ is a branch point in $\Gamma^n_{\textrm{Fix}}$ there must be some finite direction $\vv_a\in T_P(\Gamma^n_{\textrm{Fix}})$ with $P\in B_{(\phi^n)_*\vv_a}(Q)^- = \phi^n(B_{\vv_a}(P)^-)$. This implies that $(\phi^n)_* \vv_a = \vw_{QP}$, and so $\infty \in \phi^n(B_{\vv_a}(P)^-)$. This contradicts that $\phi^n$ does not have a pole in any finite direction at $P$, and so we conclude that $Q\in (P, \infty]$. Write $Q=\zeta_{0, s}$, $s>r$.

We next claim that for any finite direction $\vv_a\in T_P(\Gamma^n_{\textrm{Fix}})$, $(\phi^n)_* \vv_a = \vw_0$, where $\vw_0\in T_Q$ is the direction towards $0$. If $f_a\in B_{\vv_a}(P)^-$ is a type I $n$-periodic point, then $f_a\in \phi^n(B_{\vv_a}(P)^-) = B_{(\phi^n)_*\vv_a}(Q)^-$. Since $|f_a|\leq r < s = \diam_{\zetaG}(Q)$, we must have $(\phi^n)_* \vv_a = \vw_0$, where $\vw_0\in T_Q$ is the direction towards 0. \\

As above, let $\vv_a \in T_P(\Gamma^n_{\textrm{Fix}})\setminus\{\vv_{\infty}\}$, and let $U_a=B_{\vv_a}(P)^-$. Then $\phi^n(U_a) = B_{\vw_0}(Q)^-$, and hence $\overline{U_a}\subseteq \phi^n(U_a)$. The repelling fixed point critieria (\cite{RLPP} Proposition 9.3, see also \cite{BR} Theorem 10.83) implies that each $U_a$ contains some repelling $n$-periodic point (of type I or of type II). In particular, $U_a \cap \mathcal{J}(\phi)\neq \emptyset$.

Let $\vv_b\in T_P(\Gamma^n_{\textrm{Fix}})\setminus \{\vv_\infty, \vv_0\}$. The fact that $0\in\phi^n(U_b) = B_{\vv_0}(Q)^-$ implies that $\phi^n$ has a non-zero root in $B_{\vv_b}(P)^-$. By the non-Archimedean Rolle's Theorem (see \cite{XF2} Application 1), $\phi^n$ has a critical point in the disc $D(0, r\cdot \gamma_p)$, where $\gamma_p = |p|^{-1/(p-1)}$.

If $r$ satisfies $$r< B_\phi \cdot \gamma_p^{-1} \cdot \mathcal{L}_\phi^{-2(n-1)}\ ,$$ then by Lemma~\ref{lem:critptfatou} we find $D(0,r)^-\subseteq \mathcal{F}(\phi)$. Hence the convex hull $B_{\vv_0}(P)^-$ of $D(0,r)^-$ lies in $\mathcal{F}(\phi)$. Since $\vv_0\in T_P(\Gamma^n_{\textrm{Fix}})$, this contradicts $B_{\vv_0}(P)^-\cap \mathcal{J}(\phi) = (\phi)U_0 \cap \mathcal{J}(\phi)\neq \emptyset$, and so we conclude $$r\geq B_\phi \gamma_p^{-1} \mathcal{L}_\phi^{-2(n-1)}\ ;$$ after taking logarithms we obtain the asserted bound.
\end{proof}

\begin{proof}[Proof of Theorem~\ref{thm:boundoncrucialset}]
If $\phi$ has potential good reduction and $P$ is the point where $\phi$ attains good reduction, then the first assertion of the theorem follows immediately from Proposition~\ref{prop:goodreductionbound}. If $\phi$ does not have potential good reduction, then for each point $P$ in the crucial set one of the following holds: either

\begin{enumerate}
\item $P$ is a focused repelling periodic point, or $P$ has a shearing direction. Then by Propositions~\ref{prop:focusedrepellingbound} and~\ref{prop:weightedremotebound} we have
\begin{align*}
\rho(P, \zetaG) &\leq n\log_v \mathcal{L}_\phi\ . \\
\end{align*}

\item $P$ is fixed by $\phi^n$ but has no shearing and is not a focused repelling point, or that $P$ is moved by $\phi$. Then by Propositions~\ref{prop:focusedrepellingbound} and~\ref{prop:weightedbranchpt} we have

\begin{align*}
\rho(P, \zetaG) & \leq 2(n-1) \log_v \mathcal{L}_\phi - \log_v B_\phi + \frac{1}{p-1}\ .
\end{align*}

\end{enumerate}

By taking maxima, the theorem follows.

\end{proof}
With Theorem~\ref{thm:boundoncrucialset}, we can readily establish Theorem~\ref{thm:bounds}:
\begin{proof}[Proof of Theorem~\ref{thm:bounds}]
If $\phi$ has potential good reduction and $\Phi$ is a normalized lift of $\phi$ at $\zetaG$, then it is enough to choose $N_0$ so that $$\frac{2}{d-1} \log_v |\Res(\Phi)|^{-1} < 3 n \log_v \mathcal{L}_\phi$$ for all $n\geq N_0$. 

If $\phi$ does not have potential good reduction, let $\tilde{\mathcal{L}}_{\phi}>\mathcal{L}_\phi$ denote the constant from the statement of the theorem. The bound in Theorem~\ref{thm:boundoncrucialset} still holds if we replace $\mathcal{L}_\phi$ by $\tilde{\mathcal{L}}_\phi$. 

Since $\tilde{\mathcal{L}}_\phi>1$, we may choose $N_0$ sufficiently large so that  $$\max\left(n\log_v \tilde{\mathcal{L}}_\phi, 2(n-1) \log_v \tilde{\mathcal{L}}_\phi - \log_v B_\phi +\frac{1}{p-1}\right) < 3n\log_v \tilde{\mathcal{L}}_\phi$$ for $n\geq N_0$, where $B_\phi$ is the constant from Theorem~\ref{thm:boundoncrucialset}. This, together with Theorem~\ref{thm:boundoncrucialset}, establishes the asserted bound.
\end{proof}

\section{Logarithmic Equidistribution}\label{sect:logeq}
In this section, we use the bounds derived in the preceeding section to establish that the potential functions $u_{\nu_{\phi^n}}(z, \zeta)$ converge uniformly to $u_\phi(z, \zetaG)$ on $\pberk$. 

Fix $\zeta\in \pberk$. If $\zeta$ is `close' to $\zetaG$, then we may aply the standard equidistribution in \cite{KJ} Theorem 4 to guarantee that $u_\phi(\zeta, \zetaG)$ converges at least locally uniformly.  If $\zeta$ is `far' from $\zetaG$, then for an appropriately chosen $\epsilon >0$ we can push $\zeta$ to $\zeta_\epsilon$, where $\zeta_\epsilon$ is the unique point on the path $[\zeta, \zetaG]$ with $\diamG(\zeta_\epsilon) = \epsilon$.  Consider the following decomposition:
\begin{align}
\int \log_v \delta(z, \zeta)_{\zetaG} d(\nu_{\phi^n} - \mu_\phi)(z) & = \int \log_v \delta(z, \zeta)_{\zetaG} - \log_v \delta(z, \zeta_\epsilon)_{\zetaG} d\nu_{\phi^n}(z)\label{eq:logdecompfirst} \\
& + \int \log_v \delta(z, \zeta_\epsilon)_{\zetaG} d(\nu_{\phi^n} - \mu_\phi)(z)\label{eq:logdecompsecond}\\
& + \int \log_v \delta(z, \zeta_\epsilon)_{\zetaG} - \log_v \delta(z, \zeta)_{\zetaG} d\mu_\phi(z)\label{eq:logdecompthird}\ .
\end{align}

If $\epsilon$ is chosen sufficiently small, the bounds in the preceeding section ensure that there is no crucial mass near $\zeta$ or $\zeta_\epsilon$ and so (\ref{eq:logdecompfirst}) is 0. The second term can be bounded using \cite{KJ} Theorem 4 with a bound depending only on $\epsilon$ and $n$. Finally, Proposition~\ref{prop:telescopingprop} below guarantees that (\ref{eq:logdecompthird}) is bounded in terms of $\epsilon$ and a constant depending only on $\phi$.

\subsection{Preliminary Lemmas}

Our first lemma gives an explicit bound on the integral (\ref{eq:logdecompsecond}):
\begin{lemma}\label{lem:explicitequidistribution}
Fix $\zeta \in \hberk$. There exists a constant $C_\phi$ depending only on $\phi$ so that for each $n\geq 1$, we have $$\left|\int \log_v \delta(z, \zeta)_{\zetaG} d(\nu_{\phi^n} - \mu_\phi)(z)\right| \leq \frac{4C_\phi + 12\rho(\zeta, \zetaG)}{d^n-1}\ .$$
\end{lemma}

\begin{proof}
Note that the integrand $f(z):= \log_v \delta(z, \zeta)_{\zetaG}$ is in $\CPA(\Gamma)$ for $\Gamma = [\zeta, \zetaG]$. The corresponding integrals have been bounded in \cite{KJ} Theorem 4. Using that $\Gamma = [\zeta, \zetaG]$ we can make the error terms in \cite{KJ} Theorem 4 explicit:
\begin{itemize}
\item $|\Delta| (f) = |\delta_{\zeta} - \delta_{\zetaG}| \leq 2$.
\item $\max_{\Gamma} |f| = \max_{[\zeta, \zetaG]} |\log_v \delta(z, \zeta)_{\zetaG} | = \rho(\zeta, \zetaG)$.
\item $R_\Gamma = \rho(\zeta, \zetaG)$, the radius of a ball for which $\Gamma \subseteq B(\zetaG, R_\Gamma)$.
\item $D_\Gamma = 4$. Recall that $D_\Gamma$ was computed in \cite{KJ} Lemma 11 as $$D_\Gamma = K(\Gamma) \cdot \left( \sum_{P\in \Gamma} (v(P) -2) + (E_\Gamma +1) \max_{P\in \Gamma} v(P)\right)\ .$$ Here, $v(P) = 2$ for each interior point of $\Gamma=[a^\epsilon, \zetaG]$ and $v(P) = 1$ for each endpoint. The constant $E_\Gamma$ counts the number of edges in $\Gamma$ (introduced in \cite{KJ} Proposition 3), which in our case is 1. Finally, $K(\Gamma)$ counts the number of connected components that can arise by removing a connected subgraph $\Gamma_0\subseteq \Gamma$, which for a segment can be taken as $K(\Gamma) = 2$ (see \cite{KJ} Lemma 9). Taking this together, we find that $D_\Gamma= 4$. 
\end{itemize}

Putting these estimates together, \cite{KJ} Theorem 4 implies $$\left|\int \log_v \delta(z, a^\epsilon)_{\zetaG} d(\nu_{\phi^n} -\mu_\phi)(z) \right| \leq \frac{4C_\phi + 12 \rho(\zeta, \zetaG)}{d^n-1} \ .$$

\end{proof}

Next we give an explicit bound for (\ref{eq:logdecompthird}). To do this, we will need a technical lemma modelled on a result of Favre and Rivera-Letelier (\cite{FRLErgodic} Proposition 3.3):

\begin{lemma}
Let $\nu$ be a Borel measure with H\"older continuous potentials (with respect to the small metric $d$), and let $M, \alpha$ denote the H\"older constant and exponent (resp.) for $u_\nu(z, \zetaG)$. Let $\zeta=\zeta_{a,r}$ with $r\in \left(0, \frac{1}{q_v}\right)$, and let $\vv_{\zetaG}\in T_\zeta$ denote the direction towards $\zetaG$. Then for any $\vv\in T_\zeta\setminus \{\vv_{\zetaG}\}$, we have $$\nu_\phi(B_{\vv}(\zeta)^-) \leq M (q_v-1)^\alpha r^\alpha\ .$$ In particular, $\nu$ does not charge type I points.
\end{lemma}
\begin{proof}
Let $\chi(z) = -\log_v \delta(z, \zeta_{a, r})_{\zeta_{a, q_v r}}$ be the potential function for the measure $\delta_{\zeta_{a, q_v r}}-\delta_{\zeta_{a, r}}$. Note that this function is identically equal to 1 on $B_{\vv} (\zeta)^-$ for each $\vv\in T_\zeta \setminus \{\vv_{\zetaG}\}$ and is identically equal to 0 on $\pberk \setminus B_{\vv_{\zeta_{a,r}}}(\zeta_{a, q_v r})^-$. Thus for any direction $\vv\in T_\zeta\setminus \{\vv_{\zetaG}\}$ we have the estimate 
\begin{align}
\nu(B_{\vv}(\zeta_{a,r})^-) \leq \int \chi d\Delta u_{\nu}(\cdot, \zetaG) &= \int u_\nu(z, \zetaG) d\Delta \chi\nonumber \\
&  = u_\nu(\zeta_{a, q_v r}, \zetaG) - u_\nu(\zeta_{a,r}, \zetaG)\nonumber \\
&  \leq M \textbf{d}_{\pberk}(\zeta_{a, q_v r}, \zeta_{a,r})^\alpha\label{eq:potentialHolderfirstest}\ ,
\end{align} where here we have used the fact that $M, \alpha$ are the H\"older constant and exponent (resp.) of $u_\nu(\cdot, \zetaG)$ in the $\textbf{d}_{\pberk}$-metric. We can estimate $\textbf{d}_{\pberk}(\zeta_{a, q_v\cdot r}, \zeta_{a, r})$ by considering the case $|a|\leq 1$ and $|a|>1$. In the former case, $\zeta_{a, q_v\cdot r}, \zeta_{a, r}$ lie in the same connected component $B_{\vv}(\zetaG)^-$ for some $\vv\in T_{\zetaG}\setminus \{\vv_\infty\}$, hence $$\textbf{d}_{\pberk}(\zeta_{a, q_v\cdot r}, \zeta_{a,r}) = \diamG(\zeta_{a, q_v\cdot r}) - \diamG(\zeta_{a,r}) = q_v\cdot r - r\ .$$ If $|a|>1$, then the inversion map $\iota(z) = \frac{1}{z}$ sends $\zeta_{a, q_v r}$ to $\zeta_{\frac{1}{a},\frac{q_v\cdot r}{|a|^2} }$ and $\zeta_{a, r}$ to $ \zeta_{\frac{1}{a},\frac{r}{|a|^2}}$ (resp.). Since $\textbf{d}_{\pberk}$ is $\PGL_2(\mathcal{O})$-invariant, we find $$\textbf{d}_{\pberk}(\zeta_{a, q_v r}, \zeta_{a, r}) = \textbf{d}_{\pberk}\left(\zeta_{\frac{1}{a},\frac{q_v\cdot r}{|a|^2} }, \zeta_{\frac{1}{a},\frac{r}{|a|^2} }\right) = \frac{1}{|a|^2} (q_v r -r )\leq q_v r - r\ .$$

In either case, then $\textbf{d}_{\pberk}(\zeta_{a, q_v r}, \zeta_{a,r}) \leq q_v r - r$. Inserting this into (\ref{eq:potentialHolderfirstest}) we have 

\begin{align*}
\nu(B_{\vv}(\zeta_{a,r})^-) &\leq M\textbf{d}_{\pberk}(\zeta_{a, q_v r}, \zeta_{a,r})^\alpha\\
&  \leq M (q_v r - r)^\alpha\\
& = M (q_v-1)^\alpha r^\alpha\ ,
\end{align*} which is the asserted inequality.

\end{proof}

\begin{prop}\label{prop:telescopingprop}
Let $\nu$ denote a Borel measure with H\"older continuous potentials (with respect to the small metric $\textbf{d}_{\pberk}$), and let $M, \alpha$ denote the H\"older constant and exponent (resp.) for $u_\nu(z, \zetaG)$. 

Fix $\zeta\in \pberk$ with $\diam_{\zetaG}(\zeta) \in [0, \frac{1}{q_v})$. Let $\epsilon \in \left[\diam_{\zetaG}(\zeta),\frac{1}{q_v}\right)$, and let $\zeta_\epsilon$ denote the unique point on $[\zeta, \zetaG]$ with $\diamG(\zeta_\epsilon) = \epsilon$.  Then for each $\zeta\in \pberk$, $$\left|\int \log_v \delta(z, \zeta_\epsilon)_{\zetaG} - \log_v \delta(z, \zeta)_{\zetaG} d\nu(z)\right| \leq \frac{M(q_v-1)^\alpha}{\alpha|\ln(q_v)|} (\epsilon^\alpha - \diamG(\zeta)^\alpha )\ .$$
\end{prop}
\begin{proof}
Let $\vv\in T_{\zeta_\epsilon}$ be the direction towards $\zeta$, and note that the integrand is zero on $U=\pberk \setminus B_{\vv}(\zeta_\epsilon)^-$, so we must estimate $\left|\int_{U} \log_v \delta(z, \zeta_\epsilon)_{\zetaG} - \log_v \delta(z, \zeta)_{\zetaG} d\nu(z)\right|\ .$ Let $\epsilon = \epsilon_0 >\epsilon_1 > \epsilon_2 > ... > \epsilon_N = \diamG(\zeta)$ be a partition of the interval $[\diamG(\zeta), \epsilon]$. For each $k=0,1,..., N-1$, let $\zeta_k$ denote the point on the segment $[\zeta, \zeta_\epsilon]$ with $\diamG(\zeta_k) = \epsilon_k$. Let $\vv_k\in T_{\zeta_k}$ denote the unique direction towards $\zeta_{k+1}$. We will sometimes also write $\diamG(\zeta) = \epsilon_\zeta$.\\

We begin by rewriting the integral $\int_U \log_v \delta(z, \zeta_\epsilon)_{\zetaG} - \log_v \delta(z, \zeta)_{\zetaG} d\nu(z)$ as a telescoping sum $$S(\{\epsilon_k\}):= \sum_{k=0}^{N-1} \displaystyle \int_U \log_v \delta(z, \zeta_k)_{\zetaG} - \log_v \delta(z, \zeta_{k+1})_{\zetaG} d\mu_\phi(z)\ .$$ Each integrand is bounded above by $\log_v \epsilon_{k+1} - \log_v \epsilon_k$ on the ball $B_{\vv_k}(\zeta_k)^-$, and is constant off of this ball.
In particular, we have \begin{equation}\label{eq:SNestmeasure}S(\{\epsilon_k\}) \leq \sum_{k=0}^{N-1} (\log_v \epsilon_{k+1} - \log_v \epsilon_k )\nu(B_{\vv_k}(\zeta_k)^-)\ .\end{equation} By the preceeding lemma,  $$\nu(B_{\vv_k}(\zeta_k)^-) \leq M (q_v-1)^\alpha\epsilon_k^{\alpha}$$ where $M, \alpha$ are the H\"older constant and exponent for $\phi$ with respect to the small metric $\textbf{d}_{\pberk}$. Inserting this into (\ref{eq:SNestmeasure}) gives \begin{equation*}S(\{\epsilon_k\}) \leq M(q_v-1)^\alpha\sum_{k=0}^{N-1} (\log_v \epsilon_{k+1} - \log_v\epsilon_k) \epsilon_k^\alpha\ .\end{equation*}  Making a change of variables $\eta_k = \epsilon_k^\alpha$, we find \begin{equation}\label{eq:SNestnomeasure}S(\{\epsilon_k\}) \leq \frac{M(q_v-1)^\alpha}{\alpha}\sum_{k=0}^{N-1} (\log_v \eta_{k+1} - \log_v\eta_k) \eta_k\ .\end{equation} Applying summation by parts to the above sum gives \begin{equation}\label{eq:sbp}\sum_{k=0}^{N-1} (\log_v \eta_{k+1} - \log_v \eta_k) \eta_k = (\log_v \eta_{N}) \eta_{N} - (\log_v \eta_{0}) \eta_0 - \sum_{k=0}^{N-1} \log_v \eta_{k+1} (\eta_{k+1} - \eta_k)\ .\end{equation}

Suppose now that $\diamG(\zeta) >0$. Let $||\epsilon_k|| = \sup_{k=1,..., N} (\epsilon_{k}-\epsilon_{k-1})$ denote the mesh of the partition $\{\epsilon_k\}$. Taking the limit as $||\epsilon_k||\to 0$, the expression in (\ref{eq:sbp}) becomes a definite integral:

\begin{align}
\lim_{||\eta_k||\to 0} \sum_{k=0}^{N-1} (\log_v \eta_{k+1} - \log_v \eta_k)\eta_k &= (\alpha\cdot  \log_v \epsilon_\zeta)\epsilon_\zeta^\alpha - (\alpha\cdot \log_v \epsilon) \epsilon^\alpha - \int_{\epsilon^\alpha}^{\epsilon_\zeta^\alpha} \log_v x\  dx\nonumber \\
& =\alpha((\log_v \epsilon_\zeta)\epsilon_\zeta - (\log_v \epsilon)\epsilon) - \left. \left(x\log_v x-\frac{1}{\ln(q_v)} x\right)\right|_{\epsilon^\alpha}^{\epsilon_\zeta^\alpha}\nonumber\\
& = \frac{1}{\ln(q_v)} (\epsilon^\alpha - \epsilon_\zeta^\alpha)\label{eq:actualintegral}\ .
\end{align}
Using the fact that $||\epsilon_k|| \to 0$ if and only if $||\eta_k|| \to 0$, the estimates in (\ref{eq:SNestnomeasure}), (\ref{eq:sbp}) and (\ref{eq:actualintegral}) give 
\begin{align*}
\left|\int_U \log_v \delta(z, \zeta_\epsilon)_{\zetaG} - \log_v \delta(z, \zeta)_{\zetaG} d\nu(z)\right| &= \lim_{||\epsilon_k||\to 0} |S(\{\epsilon_k\})|\\
& \leq \left|\frac{M(q_v-1)^\alpha}{\alpha} \lim_{||\eta_k||\to 0} \sum_{k=0}^{N-1} (\log_v \eta_{k+1} - \log_v \eta_k) \eta_k\right|\\
& =\frac{M(q_v-1)^\alpha}{\alpha|\ln(q_v)|} (\epsilon^\alpha - \epsilon_\zeta^\alpha)\ .
\end{align*}

In the case that $\diamG(\zeta) = 0$, for every $\delta>0$ let $\zeta_\delta$ be the unique point on $[\zeta, \zeta_\epsilon]$ with $\diamG(\zeta_\delta)=\delta$. If we take partitions $\{\epsilon_k\}$ of the smaller interval $[\delta, \epsilon]$, the above estimates imply 
\begin{align}
\left|\int_U \log_v \delta(z, \zeta_\epsilon)_{\zetaG} - \log_v \delta(z, \zeta_\delta)_{\zetaG} d\nu(z)\right| &= \lim_{||\epsilon_k||\to 0} |S(\{\epsilon_k\})|\nonumber \\
& \leq \left|\frac{M(q_v-1)^\alpha}{\alpha} \lim_{||\eta_k||\to 0} \sum_{k=0}^{N-1} (\log_v \eta_{k+1} - \log_v \eta_k) \eta_k\right|\nonumber \\
& =\frac{M(q_v-1)^\alpha}{\alpha|\ln(q_v)|} (\epsilon^\alpha - \delta^\alpha)\label{eq:deltaintegralestimate}\ .
\end{align}
The integrand $\log_v \delta(z, \zeta_\epsilon)_{\zetaG} - \log_v \delta(z, \zeta_\delta)_{\zetaG}$ is non-negative on $U$ and is non-decreasing as $\delta \to 0$. By the monotone convergence theorem, taking the limit as $\delta\to 0 = \diamG(\zeta)$ and applying (\ref{eq:deltaintegralestimate}) gives
\begin{align*}
\left|\int_U \log_v \delta(z, \zeta_\epsilon)_{\zetaG} - \log_v \delta(z, \zeta)_{\zetaG} d\nu(z)\right| &= \lim_{\delta\to 0}\left| \int_U \log_v \delta(z, \zeta_\epsilon)_{\zetaG} - \log_v \delta(z, \zeta_\delta)_{\zetaG} d\nu(z)\right|\\
& \leq\lim_{\delta\to 0}\ \frac{M(q_v-1)^\alpha}{\alpha|\ln(q_v)|}\cdot (\epsilon^\alpha - \delta^\alpha)\\
& = \frac{M(q_v-1)^\alpha}{\alpha|\ln(q_v)|} \epsilon^\alpha
\end{align*} which is the asserted bound in the case $\diamG(\zeta) = 0$.

\end{proof}

We can finally piece this together to obtain the logarithmic equidistribution:

\begin{thm}\label{thm:quantlogeq}
Let $K$ be a complete, algebraically closed non-Archimedean valued field of characteristic 0. Let $\phi\in K(z)$ be a rational map of degree $d\geq 2$ and suppose that $\phi$ has bad reduction. Let $\mathcal{L}_\phi>1$ denote a Lipschitz constant for the action of $\phi$ on $\mathbb{P}^1(K)$ in the chordal metric and let $C_\phi$ be the constant from Lemma~\ref{lem:explicitequidistribution}. Let $M, \alpha$ be the H\"older constant and exponent (resp.) for the potential function $u_\phi(z, \zetaG)$ with respect to the small metric $\textbf{d}_{\pberk}$. For $n$ sufficiently large and for any $\zeta\in \pberk$, we have \begin{equation}\label{eq:technicalbound}\left|\int \log_v \delta(z, \zeta)_{\zetaG} d(\nu_{\phi^n} - \mu_\phi)\right| \leq \frac{36n\log_v \mathcal{L}_\phi+4C_\phi}{d^n-1} + \frac{M(q_v-1)^\alpha}{\alpha |\ln(q_v)|} (\mathcal{L}_\phi^{-n\alpha}- \diamG(\zeta)^\alpha)\ .\end{equation}

\end{thm}

\begin{proof}

 By Theorem~\ref{thm:boundoncrucialset}, we find that a point $P$ with $w_{\phi^n}(P)>0$ for some $n$ must satisfy $$\rho(P, \zetaG) \leq \max\left(n \log_v \mathcal{L}_\phi, 2(n-1) \log_v \mathcal{L}_\phi - \log_v B_\phi - \frac{1}{p-1}\right) \ .$$ We can find a constant $N_0 = N_0(\phi)$ such that, for $n\geq N_0$ and $w_{\phi^n}(P) >0$, we have \begin{equation*}\rho(P, \zetaG) \leq 3n\log_v \mathcal{L}_\phi\ .\end{equation*} Increasing $N_0$ if necessary, we can also assume that $\mathcal{L}_\phi^{-3n} < \frac{1}{q_v}$ for $n\geq N_0$; note that this additional constraint depends only on $\phi$ and $K$.

Fix $\zeta\in \hberk$ and $n\geq N_0$. If $\rho(\zeta, \zetaG) \leq 3n\log_v \mathcal{L}_\phi$, then we may apply Lemma~\ref{lem:explicitequidistribution} to find \begin{equation}\label{eq:integralboundnear}\left|\int \log_v \delta(z, \zeta)_{\zetaG}d(\nu_{\phi^n} - \mu_\phi)\right|  \leq \frac{4C_\phi + 36n \log_v \mathcal{L}_\phi}{d^n-1}\ , \end{equation}which is stronger than the  bound in (\ref{eq:technicalbound}).\\

If $\rho(\zeta, \zetaG) > 3n \log_v \mathcal{L}_\phi$, then let $\zeta_\epsilon$ denote the point on the path $[\zeta, \zetaG]$ with $\rho(\zeta_\epsilon, \zetaG) = 3n \log_v \mathcal{L}_\phi$. More explicitly, $\epsilon :=\diamG(\zeta_\epsilon) = \mathcal{L}_\phi^{-3n}$. Recalling the decomposition given above, we rewrite our integral as 
\begin{align}
\int \log_v \delta(z, \zeta)_{\zetaG} d(\nu_{\phi^n} - \mu_\phi) & = \int \log_v \delta(z, \zeta)_{\zetaG} - \log_v \delta(z, \zeta_\epsilon)_{\zetaG} d\nu_{\phi^n}\tag{\ref{eq:logdecompfirst}} \\
& + \int \log_v \delta(z, \zeta_\epsilon)_{\zetaG} d(\nu_{\phi^n} - \mu_\phi)\tag{\ref{eq:logdecompsecond}}\\
& + \int \log_v \delta(z, \zeta_\epsilon)_{\zetaG} - \log_v \delta(z, \zeta)_{\zetaG} d\mu_\phi\tag{\ref{eq:logdecompthird}}\ .
\end{align}

Since $\rho(\zeta, \zetaG), \rho(\zeta_\epsilon, \zetaG) \geq 3n \log_v \mathcal{L}_\phi$, Theorem~\ref{thm:boundoncrucialset} guarantees that $\nu_{\phi^n}$ does not charge the segment $[\zeta, \zeta_\epsilon]$, hence (\ref{eq:logdecompfirst}) is zero. Applying Lemma~\ref{lem:explicitequidistribution} to (\ref{eq:logdecompsecond}), we find that $$\left|\int \log_v \delta(z, \zeta_\epsilon)d(\nu_{\phi^n}-\mu_\phi)\right| \leq \frac{4C_\phi + 36n\log_v \mathcal{L}_\phi}{d^n-1}\ .$$ Finally, by Proposition~\ref{prop:telescopingprop} the term (\ref{eq:logdecompthird}) can be bounded as $$\left|\int \log_v \delta(z, \zeta_\epsilon)_{\zetaG} - \log_v \delta(z, \zeta)_{\zetaG} d\mu_\phi\right| \leq \frac{M(q_v-1)^\alpha}{\alpha |\ln(q_v)|} (\mathcal{L}_\phi^{-n\alpha} - \diamG(\zeta)^\alpha)\ .$$ Combining these gives
\begin{equation}\label{eq:integralboundfar}
\left|\int \log_v \delta(z, \zeta)_{\zetaG} d(\nu_{\phi^n} - \mu_\phi)\right| \leq \frac{36n\log_v \mathcal{L}_\phi+4C_\phi}{d^n-1} + \frac{M(q_v-1)^\alpha}{\alpha |\ln(q_v)|} (\mathcal{L}_\phi^{-n\alpha}- \diamG(\zeta)^\alpha)
\end{equation}

The inequalities in (\ref{eq:integralboundnear}) and (\ref{eq:integralboundfar}) imply the bound asserted in the statement of the theorem.

\end{proof}

We can apply the preceeding technical theorem to prove Theorem~\ref{thm:logeq} asserted in the introduction:

\begin{proof}[Proof of Theorem~\ref{thm:logeq}]
Fix $\epsilon>0$, and choose $n\gg 0$ so that Theorem~\ref{thm:quantlogeq} applies and such that $$\frac{36n\log_v \mathcal{L}_\phi + 4C_\phi}{d^n-1} + \frac{M(q_v-1)^\alpha}{\alpha |\ln(q_v)|} (\mathcal{L}_\phi^{-n\alpha}) < \epsilon\ .$$ For any $\zeta_0\in \hberk$ and $z,\zeta\in \pberk$. Recall that $$\delta(z,\zeta)_{\zeta_0} = \frac{\delta(z,\zeta)_{\zetaG}}{\delta(z,\zeta_0)_{\zetaG} \delta(\zeta, \zeta_0)_{\zetaG}}\ .$$Thereofre, it sufficies to prove the result for $\zeta_0=\zetaG$. Theorem~\ref{thm:quantlogeq} implies that for any $\zeta\in \pberk$, we have 
\begin{align*}
 \left| \int \log_v \delta(z, \zeta)_{\zetaG} d(\nu_{\phi^n} - \mu_\phi)\right| &< \frac{36n\log_v \mathcal{L}_\phi + 4C_\phi}{d^n-1} + \frac{M(q_v-1)^\alpha}{\alpha |\ln(q_v)|} (\mathcal{L}_\phi^{-n\alpha}- \diamG(\zeta)^\alpha)\\
& \leq  \frac{36n\log_v \mathcal{L}_\phi + 4C_\phi}{d^n-1} + \frac{M(q_v-1)^\alpha}{\alpha |\ln(q_v)|} \mathcal{L}_\phi^{-n\alpha}\\
& < \epsilon\ .
\end{align*}
\end{proof}

As an application, we have a proof of Corollary~\ref{cor:potentialfunctionconvergences}:

\begin{proof}[Proof of Corollary~\ref{cor:potentialfunctionconvergences}]
The first result is essentially a restatement of the convergence given in Theorem~\ref{thm:logeq}, for $$u_{\nu_{\phi^n}}(z, \zeta) = -\int \log_v \delta(w, z)_{\zeta} d\nu_{\phi^n}(w)\ .$$

For the second result, we first claim that \begin{equation}\label{eq:iteratedintegralconvergence}\iint \log_v ||x,y|| d\nu_{\phi^n}(x)d\nu_{\phi^n}(y) \to \iint \log_v ||x,y|| d\mu_\phi(x) d\mu_\phi(y)\ .\end{equation} Fix $\epsilon >0$. The integrals above can be rewritten in terms of the respective potential functions $u_{\nu}(\cdot, \zetaG)$; more precisely, for any Borel measure $\nu$ we have $\iint \log_v ||x,y|| d\nu(x)d\nu(y) = -\int u_{\nu}(y, \zetaG) d\nu(y)$. 

By the uniform convergence of the potential functions $u_{\phi^n}(\cdot, \zetaG)$, we may choose $N_0$ sufficiently large so that \begin{equation}\label{eq:closepotentials}|u_{\phi^n}(y, \zetaG) - u_\phi(y, \zetaG)|< \epsilon\end{equation} for all $n\geq N_0$. Since $u_{\phi^n}(\cdot, \zetaG)$ is continuous on $\pberk$, by \cite{KJ} Theorem 2 we may increase $N_0$ if necessary to ensure that \begin{equation}\label{eq:closepotentialintegrals}\left|\int u_{\phi^n}(y, \zetaG) - u_\phi(y, \zetaG) d(\nu_{\phi^n} - \mu_\phi)(y)\right|< \epsilon\end{equation} for $n\geq N_0$. Combining (\ref{eq:closepotentials}) and (\ref{eq:closepotentialintegrals}) establishes the claim: for $n\geq N_0$,
\begin{align*}
\left|\iint \log_v ||x,y|| d\nu_{\phi^n}(x)d\nu_{\phi^n}(y) -\right. & \left. \iint \log_v ||x,y|| d\mu_\phi(x)d\mu_\phi(y)\right| \\
& = \left|\int u_{\nu_{\phi^n}}(y, \zetaG) d\nu_{\phi^n}(y) - \int u_\phi(y, \zetaG) d\mu_\phi(y)\right|\\
& \leq \left|\int u_{\nu_{\phi^n}}(y, \zetaG) - u_\phi(y, \zetaG) d\nu_{\phi^n}\right| \\
& \hspace{0.1cm}+ \left| \int u_\phi(y, \zetaG) d(\nu_{\phi^n} - \mu_\phi)(y)\right|\\
& < 2\epsilon\ .
\end{align*}

We now show the uniform convergence of the two-variable Arakelov-Green's functions $g_{\nu_{\phi^n}}(x,y)$. For any probability measure, the Arakelov-Green's function admits the decomposition $$g_{\nu}(x,y) = -\log_v ||x,y|| + u_{\nu}(x, \zetaG) + u_{\nu}(y, \zetaG) + C_\nu\ .$$ The convergence of the potential functions follows from part 1 of the Corollary. We need to only show that the constants $C_{\nu_{\phi^n}}$ converge to the constant $C_\phi$ associated to $g_\phi(x,y)$. These constants are given explicitly by $$C_{\nu} = -\iint g_\nu(x,y) d\nu(x)d\nu(y) = \iiint \log_v \delta(x,y)_\zeta d\nu(\zeta) d\nu(x)d\nu(y)\ .$$ This latter integral can be decomposed as \begin{align*}
\iiint \log_v \delta(x,y)_\zeta d\nu(\zeta)d\nu(x)d\nu(y) &= \iint \log_v ||x,y||d\nu(x)d\nu(y) - \iint \log_v ||x,\zeta||d\nu(\zeta)d\nu(x)\\
& \hspace{1cm}- \iint \log_v ||y, \zeta|| d\nu(y)d\nu(\zeta)\ .
\end{align*} Thus the convergence of the $C_{\nu_{\phi^n}}$ to $C_\phi$ follows from (\ref{eq:iteratedintegralconvergence}), and hence $g_{\nu_{\phi^n}}(x,y)$ converges uniformly to $g_\phi(x,y)$.

For the third assertion, we rely on a result of Okuyama: by \cite{KJ} Theorem 2 the measures $\nu_{\phi^n}$ converge weakly to $\mu_\phi$. By the first assertion above, $u_{\nu_{\phi^n}}(c, \zetaG) \to u_{\mu_\phi}(c, \zetaG)$ for each critical point $c$ of $\phi$. Then \cite{Ok} Lemma 3.1 implies \begin{equation*}\int_{\pberk} \log_v [\phi^\#] d\nu_{\phi^n} \to L_v(\phi):= \int_{\pberk} \log_v [\phi^\#] d\mu_\phi\end{equation*} as asserted.

\end{proof}

\section{Bounds on $\MinResLoc(\phi^n)$ and $\Bary(\mu_\phi)$}\label{sect:barybounds}

In this section we give explicit bounds for the distance from $\zetaG$ to $\MinResLoc(\phi^n)$ and to $\Bary(\mu_\phi)$. The main lemma used in this task is

\begin{lemma}\label{lem:coefficientlemma}
Let $\Phi$ be a normalized lift of $\phi$. Let $\Phi^n = [F, G]$ be a normalized lift for the $n$th iterate of $\phi$, where $F(X,Y) = \alpha_D X^D + ... + \alpha_0 Y^D$, $G(X,Y) = \beta_D X^D + ... + \beta_0 Y^D$ and $D=d^n$. Then \begin{align*} \max (|\alpha_0|, |\beta_0|) &\geq |\Res(\Phi)|^{\frac{d^n-1}{d-1}}\\ \max (|\alpha_{d^n}|, |\beta_{d^n}|) & \geq |\Res(\Phi)|^{\frac{d^n-1}{d-1}}\ .\end{align*}
\end{lemma}
\begin{proof}
We observe that $|\alpha_0| = |F(0,1)|, |\alpha_D| = |F(1,0)|$ and $|\beta_0| = |G(0,1)|, |\beta_D| = |G(1,0)|$. For a pair $(x,y)$, let $||(x,y)|| = \max(|x|, |y|)$. Then by \cite{BR} Lemma 10.1, we have 
\begin{align*}
\max( |\alpha_0|, |\beta_0|) = ||\Phi^n(0,1)|| & \geq ||\Phi^{n-1}(0,1)||^d\cdot |\Res(\Phi)|\\
& \geq ||\Phi^{n-2}(0,1)||^{d^2} \cdot |\Res(\Phi)|^{1+d}\\
& \dots\\
& \geq ||(0,1)|| \cdot |\Res(\Phi)|^{1+d+...+d^{n-1}} = |\Res(\Phi)|^{\frac{d^n-1}{d-1}}\ .
\end{align*}
A similar argument holds for $\max(|\alpha_{d^n}|, |\beta_{d^n}|)$. 

\end{proof}

\subsection{Bounds on $\MinResLoc(\phi^n)$ and $\Bary(\mu_\phi)$}

Lemma~\ref{lem:coefficientlemma} above gives us a bound on the size of leading and constant coefficients of the polynomials that form a normalized lift of $\phi^n$. Similar bounds appeared in the proof of \cite{Ru1} Proposition 1.8, which gave a bound for the set $\MinResLoc(\phi)$. We can use the previous lemma to strengthen this bound for iterates:

\begin{prop}\label{prop:minreslocbd}
Let $d\geq 2$ and let $R=\frac{2}{d-1} \ordRes(\phi)$. Fix a point $x\in\mathbb{P}^1(K)$. Along the segment $[\zetaG, x]$, the function $\ordRes_{\phi^n}$ satisfies \begin{equation}\label{eq:ordreslowerbound}\frac{1}{d^{2n}-d^n}\ordRes_{\phi^n}(\zeta) \geq \rho(\zetaG, \zeta)+\frac{1}{d^{2n}-d^n} \ordRes_{\phi^n}(\zetaG) - R\ .\end{equation} 

Let $\xi$ be the unique point in $[\zetaG, x]$ such that $\rho(\zetaG, \xi) = \frac{2}{d-1}\ordRes(\phi)$. Then for each $n$, the function $\ordRes_{\phi^n}(\cdot)$ is increasing along $[\xi, x]$ as one moves away from $\xi$.
\end{prop}

\begin{proof}
The proof follows \cite{Ru1} Proposition 1.8 closely. After a change of co\"ordinates by some $\gamma\in \GL_2(\mathcal{O})$, we can assume that $x=0$. Let $\Phi^n = [F, G]$ be a normalized lift of $\phi^n$, where $D=D^n$, $F(X,Y) = a_{D}X^{D} + ... + a_0 Y^{D}$, $G(X,Y) = b_{D} X^{D} + ... + b_0 Y^{D}$, where $a_i, b_j \in \mathcal{O}$ and at least one coefficient is a unit. 

Given $A\in K^\times$, let $\tau_A(z) = Az$. In \cite{Ru1} Proposition 1.8, Rumely shows that \begin{align*}
\ordRes_{\phi^n}(\zeta_{0, |A|}) -& \ordRes_{\phi^n}(\zetaG)\\
&  \geq \max \left( -2D \ord(a_0) + (D^2+D)\ord(A), -2D \ord(b_0) + (D^2-D)\ord(A),\right.\\
& \left. -2D\ord(a_D) + (D-D^2)\ord(A), -2D \ord(b_D) + (-D-D^2)\ord(A)\right)\ .
\end{align*}

\noindent Using the bounds in Lemma~\ref{lem:coefficientlemma}, this gives that 
\begin{align}
\ordRes_{\phi^n}(\zeta_{0, |A|}) -& \ordRes_{\phi^n}(\zetaG)\nonumber \\
& \geq -2D\frac{d^n-1}{d-1} \ordRes(\phi) + \max\left( (D^2-D)\ord(A), (D-D^2) \ord(A)\right)\label{eq:ordresincreasing}\ .
\end{align}

Restricting ourselves to $\ord(A)>0$, the right side of (\ref{eq:ordresincreasing}) is $-2\frac{d^{2n}-d^n}{d-1} \ordRes(\phi) + (d^{2n}-d^n) \ord(A)$.  Thus, $$\frac{1}{d^{2n}-d^n}\ordRes_{\phi^n}(\zeta_{0, |A|}) \geq  \ord(A) - \frac{2}{d-1} \ordRes(\phi) + \frac{1}{d^{2n}-d^n} \ordRes_{\phi^n}(\zetaG)$$ which establishes the first claim.

When $\ord(A)=0$, the left hand side of (\ref{eq:ordresincreasing}) is exactly equal to 0. Thus, if $\ord(A)$ is chosen large enough so that the right hand side of (\ref{eq:ordresincreasing}) is positive, the function $\ordRes_{\phi^n}(\cdot)$ must be increasing for all larger values of $\ord(A)$. This is attained for $$(D^2-D) \ord(A) \geq \frac{2D(d^n-1)}{d-1} \ordRes(\phi)\ ,$$ or equivalently, inserting the definition of $D=d^n$, $$\ord(A) \geq \frac{2}{d-1} \ordRes(\phi)\ .$$ 

\end{proof}

\begin{cor}\label{cor:boundonminresloc}
Let $\phi\in K(z)$ be a rational function of degree $d\geq 2$. Let $R=\frac{2}{d-1} \ordRes(\phi)$. Then for each $n$, $$\MinResLoc(\phi^n) \subseteq B_\rho(\zetaG, R)\ .$$ In particular, $\diam(\MinResLoc(\phi^n)) \leq \frac{4}{d-1} \ordRes(\phi)\ .$
\end{cor}

Note that this proposition and its corollary imply that the bound in Lemma~\ref{lem:coefficientlemma} is as sharp as one would expect in general. In particular, if the bound grew more slowly, say exponentially of order $n$ rather than order $d^n$, we would have the sets $\MinResLoc(\phi^n)$ converging to $\zetaG$. 

Proposition~\ref{prop:minreslocbd} can also be used to give a lower bound for the Arakelov-Green's function:

\begin{lemma}
Let $R=\frac{2}{d-1}\ordRes(\phi)$. Fix any type I point $x$. For any point $\zeta\in [\zetaG, x]$, we have $$g_{\phi}(\zeta, \zeta) \geq \rho(\zetaG, \zeta) + g_{\phi}(\zetaG, \zetaG) - R\ .$$
\end{lemma}

\begin{proof}
We use the convergence of the functions $\frac{1}{d^{2n}-d^n}\ordRes_{\phi^n}(x)$ given in \cite{KJ} Theorem 1. Let $\epsilon >0$, and fix $\zeta\in [\zetaG,x]$. We may choose $n$ large enough so that
\begin{align*}
\left|\frac{1}{d^{2n}-d^n} \ordRes_{\phi^n}(\zeta) - g_\phi(\zeta, \zeta)\right| &< \epsilon \\
\left|\frac{1}{d^{2n}-d^n} \ordRes_{\phi^n}(\zetaG) - g_\phi(\zetaG, \zetaG)\right| &< \epsilon \ .
\end{align*}

Combining this with (\ref{eq:ordreslowerbound}), we find 
\begin{align*}
g_\phi(\zeta, \zeta)+\epsilon &\geq \frac{1}{d^{2n}-d^n} \ordRes_{\phi^n}(\zeta)\\
& \geq  \rho(\zetaG, \zeta) - R + \frac{1}{d^{2n}-d^n} \ordRes_{\phi^n}(\zetaG) \\
& \geq \rho(\zetaG, \zeta) -R + g_\phi(\zetaG, \zetaG) - \epsilon\ .
\end{align*}
 Letting $\epsilon \to 0$ gives the result.

\end{proof}

We can apply this to obtain a bound on the barycenter of $\mu_\phi$:

\begin{prop}
Let $R=\frac{2}{d-1} \ordRes(\phi)$ and $m_0 = \min_{x\in \pberk} g_\phi(x,x)$. Then $$\Bary(\mu_\phi) \subseteq B_\rho(\zetaG, R+m_0 - g_\phi(\zetaG))\ .$$ We further have $$\diam(\Bary(\mu_\phi)) \leq 2 (R+m_0 - g_\phi(\zetaG))\ .$$ In particular, if we choose a co\"ordinate system so that $\zetaG\in \Bary(\mu_\phi)$, then $$\Bary(\mu_\phi) \subseteq B_\rho(\zetaG, R)$$ and $$\diam(\Bary(\mu_\phi)) \leq 2R\ .$$ 
\end{prop}

\begin{proof}
Let $R=\frac{2}{d-1}\ordRes(\phi)$, and fix $\epsilon>0$. Let $\Bary(\mu_\phi)$ be the segment $[\zeta_1, \zeta_2]$, and without loss of generality assume $\rho(\zetaG, \zeta_2) \geq \rho(\zetaG, \zeta_1)$. By the preceeding lemma, $$\rho(\zetaG, \zeta_2) \leq g_\phi(\zeta_2, \zeta_2) + R - g_\phi(\zetaG, \zetaG)\ .$$ Since $\zeta_2\in \Bary(\mu_\phi)$ and $g_\phi(x,x)$ is minimized on $\Bary(\mu_\phi)$, this gives $$\rho(\zetaG, \zeta_2) \leq R + m_0 - g_\phi(\zetaG, \zetaG)\ .$$ The last assertion follows from the fact that $g_\phi(\zetaG, \zetaG) = m_0$ if $\zetaG\in \Bary(\mu_\phi)$. 

\end{proof}

\subsection{Multipliers of Periodic Points}

Lemma~\ref{lem:coefficientlemma} can also be used to bound how repelling a type I repelling $n$-periodic point can be. More precisely, we have

\begin{prop}
Let $P$ be a type I repelling $n$-periodic point for $\phi$. Let $\Phi$ be a normalized lift for $\phi$. If $\lambda_P$ is the multiplier of $P$, we have $$|\lambda_P| \leq |\Res(\Phi)|^{-\frac{d^n-1}{d-1}}\ .$$
\end{prop}
\begin{proof}
After changing co\"ordinates by an element $\gamma\in \PGL_2(\mathcal{O})$, we may assume that $P=0$. Note that $|\Res(\Phi)|$ is unaffected by this type of conjugation. 

Let $D=d^n$ and $\phi^n(z) = \frac{f(z)}{g(z)}$, where $f(z) = a_D z^D+ ... + a_1 z$, $g(z) = b_D z^D+...+b_1 z+b_0$ are normalized, coprime polynomials representing the $n$th iterate of $\phi$. We have that $|a_1|\leq 1$ and $b_0 \neq 0$. The multiplier $\lambda_P$ is given $$\lambda_P = \frac{a_1}{b_0}\ .$$ By Lemma~\ref{lem:coefficientlemma}, we know $$\frac{1}{|b_0|} \leq |\Res(\Phi)|^{-\frac{d^n-1}{d-1}}\ .$$ Thus, $$|\lambda_P| = \frac{|a_1|}{|b_0|} \leq |\Res(\Phi)|^{-\frac{d^n-1}{d-1}}\ .$$
\end{proof}




\bibliography{EqII}

\begin{thebibliography}{10}

\bibitem{BR}
Matthew Baker and Robert Rumely.
\newblock {\em {Potential Theory and Dynamics on the Berkovich Projective
  Line}}.
\newblock AMS, 2010.

\bibitem{Ber}
Vladimir~G. Berkovich.
\newblock {\em Spectral Theory and Analytic Geometry over non-Archimedean
  Fields}.
\newblock AMS, 1990.

\bibitem{CR}
T.~Chinburg and R.~Rumely.
\newblock {The Capacity Pairing}.
\newblock {\em J. Reine Agnew. Math.}, 434:1--44, 1993.

\bibitem{XF}
Xander Faber.
\newblock {Topology and Geometry of the Berkovich Ramification Locus I}.
\newblock To appear in \emph{Manuscripta Mathematica}.

\bibitem{XF2}
Xander Faber.
\newblock {Topology and Geometry of the Berkovich Ramification Locus II}.
\newblock {\em Mathematische Annalen.}, 356:819--844, 2013.

\bibitem{FJ}
Charles Favre and Mattias Jonsson.
\newblock {\em The Valuative Tree}.
\newblock Springer-Verlag, 2004.

\bibitem{FRLErgodic}
Charles Favre and Juan Rivera-Letelier.
\newblock {Th\'eorie Ergodique des Fractions Rationelles sur un Corps
  Ultram\'etrique}.
\newblock {\em Proc. Lond. Math. Soc.}, 1:116--154, 2010.

\bibitem{KJ}
Kenneth Jacobs.
\newblock An equidistribution result for dynamical systems on $\pberk$.
\newblock {\em arxiv.org:1409.4808}, 2014.

\bibitem{Ok}
Yusuke Okuyama.
\newblock Repelling periodic points and logarithmic equidistribution in
  non-archimedean dynamics.
\newblock {\em Acta Arith.}, 152(3):267--277, 2012.

\bibitem{Prz}
Feliks Przytycki.
\newblock {Lyapunov Characteristic Exponents are Non-Negative}.
\newblock {\em Proceedings of the American Mathematical Society}, 119(1), 1993.

\bibitem{RL}
Juan Rivera-Letelier.
\newblock {Dynamique des fonctions rationnelles sur les corps locaux}.
\newblock {\em Ast\'erique}, 287:147--230, 2003.

\bibitem{RLPP}
Juan Rivera-Letelier.
\newblock Points p\'eriodiques des fonctions rationnelles dans l'espace
  hyperbolique p-adique.
\newblock {\em Comment. Math. Helv.}, 80:593--629, 2005.

\bibitem{Ru1}
Robert Rumely.
\newblock {The Minimal Resultant Locus}.
\newblock {\em arXiv.org:1304.1201}, April 2013.

\bibitem{Ru2}
Robert Rumely.
\newblock {The Geometry of the Minimal Resultant Locus}.
\newblock {\em arXiv.org:1402.6017}, Feb 2014.

\bibitem{RW}
Robert Rumely and Stephen Winburn.
\newblock {The Lipschitz Constant of a non-Archimedean Rational Function}.
\newblock In preparation.

\bibitem{ADS}
Joseph Silverman.
\newblock {\em {The Arithmetic of Dynamical Systems}}.
\newblock Springer, 2007.

\bibitem{Th}
Amaury Thuillier.
\newblock {\em Th\'eorie du potential sur les courbes en g\'eom\'etrie
  analytique non archim\'edienne. Applications \`a la th\'eorie d'Arakelov}.
\newblock Phd thesis, University of Rennes, 2005.

\end{thebibliography}
\bibliographystyle{plain}

\end{document}